\documentclass{amsart}
\usepackage{amssymb}
\usepackage{mathrsfs}

\usepackage{graphicx}
\usepackage{comment}
\newcommand{\comments}[1]{}
\usepackage[francais,english]{babel}
\markleft{uuugh}

\numberwithin{equation}{section}
\newcommand{\ve}{\varepsilon}

\newcommand{\be}{\begin{equation}}
\newcommand{\ee}{\end{equation}}
\newcommand{\ba}{\begin{align}}
\newcommand{\ea}{\end{align}}

\newcommand{\abs}[1]{\lvert#1\rvert}

\newtheorem{example}{Example}[section]
\newtheorem{theorem}{Theorem}[section]

\newtheorem{lemma}{Lemma}[section]

\newtheorem{conjecture}{Conjecture}[section]
{\begin{list}{}{%
\settowidth{\labelwidth}{\textsf{{\it #1.}}}%
\setlength{\labelsep}{4mm}%
\setlength{\leftmargin}{\labelwidth}%
\addtolength{\leftmargin}{\labelsep}%
}}%
{\end{list}}

{\begin{list}{}{%
\settowidth{\labelwidth}{\textsf{{\it #1.}}}%
\setlength{\labelsep}{2mm}%
\setlength{\leftmargin}{\labelwidth}%
\addtolength{\leftmargin}{\labelsep}%
\addtolength{\leftmargin}{4mm}%
\setlength{\itemsep}{6pt}%
\setlength{\listparindent}{0pt}%
\setlength{\topsep}{3pt}%
}}%
{\end{list}}

\title{A generalization of the Goresky-Klapper conjecture, Part I}

\author[]{Badria Alsulmi}
\address{ Department of Mathematics\\
Umm al-Quara University\\
Mecca 24382 \\ Saudi Arabia}
\email{bmsulmi@uqu.edu.sa}

\author[]{Todd Cochrane}
\address{ Department of Mathematics\\
         Kansas State University\\
         Manhattan, KS 66506 USA}
\email{cochrane@ksu.edu, pinner@ksu.edu, crichardson@ksu.edu, ianiat11@gmail.com}

\author[]{Michael J. Mossinghoff}\thanks{This work was supported in part by a grant from the Simons Foundation (\#426694 to M.~J. Mossinghoff).}
\address{Department of Mathematics \&  Computer Science \\Davidson College\\Davidson, NC 28035, USA. {\it E-mail address:} {\tt mimossinghoff@davidson.edu}}

\author[]{Vincent Pigno}
\address{ Department of Mathematics \& Statistics\\
California State University, Sacramento\\
Sacramento, CA 95819\\ USA.  {\it E-mail address:} {\tt vincent.pigno@csus.edu}}

\author[]{Chris Pinner}

\author[]{C. J.  Richardson}

\author[]{Ian  Thompson}

\date{\today}
\thanks{Part of this paper was  a thesis project for the sixth author and an undergraduate research project for the seventh author.}

\begin{document}
\selectlanguage{english}

\begin{abstract} For  a fixed integer $n\geq 2,$ we show that  a permutation of the least residues mod $p$ of the  form $f(x)=Ax^k$ mod $p$ cannot map a residue class mod $n$ to just one residue class mod $n$  once $p$ is sufficiently large, other than the maps $f(x)=\pm x$  mod $p$ when $n$ is even and $f(x)=\pm x$ or $\pm x^{(p+1)/2}$ mod $p$ when $n$ is odd.

\end{abstract}

\maketitle

\section{Introduction}\label{secIntroduction}

For an odd prime $p$ we let $I$ denote the reduced residues mod $p$,
$$ I=\{1,2,\ldots ,p-1\}, $$
and $A$ and $k$ integers  with
\be \label{range}  |A| < p/2, \;\; p\nmid A,\;\;\; 1\leq k< p-1,\;\; \gcd(k,p-1)=1, \ee
so that the map $f:I\rightarrow I$ given by
$$ f(x)= Ax^k \text{ mod } p, $$
 is a permutation of $I$.

Goresky \& Klapper \cite{GK1} divided $I$ into the even and odd residues
$$ E=\{2,4,\ldots ,p-1\},\, \;\;\; O=\{1,3,\ldots , p-2\}, $$
and asked when $f$ could  also be  a permutation of $E$ (equivalently  $O$). Originally the problem was phrased in terms of decimations of $\ell$-sequences and was restricted to cases where 2 is a primitive root mod $p,$ but this is the form that we are interested in here. Apart from the identity map $(p;A,k)=(p;1,1)$ they found six cases
$$ (p;A,k)=(5;-2,3), (7;1,5), (11;-2,3),(11;3,7), (11;5,9),(13;1,5), $$
and conjectured that there were no more for $p>13$. This was proved for sufficiently large $p$ in \cite{Bourgain} and in full in \cite{CochKony}, with asymptotic counts on $\abs{f(E)\cap O}$ considered in \cite{Bourgain2}. Since $x\mapsto p-x$ switches elements of $E$ and $O$,  this is the same as asking when $f(E)=O$ or $f(O)=E$ on replacing $A$ by $-A.$

Somewhat related is a question of Lehmer \cite[Problem F12, p. 381]{guy} concerning the number of  $x$ mod $p$ whose inverse,  $f(x)=x^{-1}$ mod $p$,  has opposite parity. Since $k$ is defined mod $(p-1)$ it is  sometimes useful to allow negative exponents, $|k|<(p-1)/2$.
This problem was solved by Zhang \cite{Wenpeng} using Kloosterman sums; see  also the generalizations by Alkan, Stan and Zaharescu \cite{Alkan}, Lu and Yi \cite{luyi1,luyi2}, Shparlinski \cite{Igor,Igor2}, Xi and Yi \cite{xiyi}, and Yi and Zhang \cite{yizhang}.

Thinking of the evens and odds as a mod 2 restriction, we can ask a similar question for a general modulus $n$. Namely we can divide
up $I$ into the $n$ congruence classes mod $n$
\begin{equation} \label{Ij}
I_j=\{x \; :\; 1\leq x\leq p-1,\; x\equiv j \hbox{ mod }n \},\;\; j=0,\ldots ,n-1,
\end{equation}
and ask for examples of the following types.

{\bf Type (i):} $f(I_j)=I_j$  for all $j=0,\ldots, n-1.$

{\bf Type (iia):}  $f(I_0),\ldots ,f(I_{n-1})$ a permutation of $I_0,\ldots ,I_{n-1}$.

{\bf Type (iib):}   $f(I_j)=I_j$  for some  $j$.

{\bf Type (iii):}  There is a pair  $i,j$ with  $f(I_i)\subseteq I_j.$

{\bf Type (iv):}  There is a pair  $i,j$ with $f(I_i)\cap I_j=\emptyset. $


In this paper we will be primarily be interested in the Type (i)-(iii) maps, though we will include some special cases of Type (iv), for example when
\be \label{defd}  d:=\gcd (k-1,p-1)   \ee
is small. We return to consider general Type (iv) in Part II.

Notice that for $n=2$  determining Type (i) through Type (iv)  are all the same problem, but for  general $n$ they can be quite different
 (indeed the $I_j$ will not even have the same cardinality unless we restrict to $p\equiv 1$ mod $n$).
 Note that these requirements become successively weaker (and the claim that there are no such examples for large enough $p$ a successively stronger statement) as we move from (i) to (iia) or (iib), to (iii), to (iv).
 To make sense here we should probably  think of $p$ growing with $n$, for example we  shall assume throughout that $p>n+1,$
otherwise all the residue classes have only 0 or 1 element and every permutation will be a Type  (iia).
Similarly  if a permutation is not a Type
 (iii),  or  Type (iv), then we are demanding at least two, or at least $n$, values in each image of each residue class and so must have $p> 2n$, or  $p>n^2,$  for this
 to have any chance of being true.

If the map $f$ randomly distributes the values mod $n$ then we might expect to have $|f(I_i)\cap I_j| \sim  p/n^2$
and  so,  for fixed $n$,  no examples of Type (i) through (iv) once  $p$ is sufficiently large.
 However, as shown in \cite{Bourgain2} for $n=2$,  if the parameter
$ d=\gcd (k-1,p-1)   $
is large we can't expect this equal distribution.

Indeed when $n$ is odd it is not hard to see that we will have infinitely many examples of Type (iib) in addition to the identity map. 

\begin{example}\label{ex1} Suppose that $p \equiv 1$ mod $4$ and that
$$f(x)=\pm x^{(p+1)/2} \text{ mod } p. $$
If $n$ is odd and $i\equiv 2^{-1}p$ mod $n$  then
$$f(I_i)=I_i. $$
If  $n$ is even, or $n$ is odd with $i \not\equiv  2^{-1}p$ mod $n$, and $p>(n+1)^2$, then
 $f(I_i)$ hits exactly two residue classes,  namely $I_i$ and $I_{\overline{i}}$
where $\overline{i}\equiv p-i$ mod $n$.
\end{example}
The proof of Example \ref{ex1}  will be given in Section \ref{ProofEx}.
At the expense of the explicit constant the   condition $p>(n+1)^2$ could be replaced  with  $p\gg (n\log n)^{4/3}$ using the Burgess \cite{Burgess} bound $O(p^{1/4}\log p)$
for gaps between quadratic residues or nonresidues.

A similar situation occurs for the map $f(x)=-x$ mod $p$; if $p>n$ and $n$ is even then the $f(I_j)=I_{\overline{j}}$ will be a derangement (i.e., a permutation fixing no element)
of the $I_j$, while  if $n$ is odd this $f$ will fix the $I_i$  with $i\equiv 2^{-1}p$ mod $n$ and derange  the remaining $I_j$.

Notice that in these examples the value of $d$ is unusually large, namely  $d=p-1$ or $(p-1)/2.$  If $d$ is not large then in fact each residue class  does receive its fair share of values:

\begin{theorem} \label{asym} For all $i,j$
\be \label{asymptotic}  |f(I_i)\cap I_j| =\frac{p}{n^2} + O(d\log^2p) +O(p^{89/92}\log^2 p). \ee
In particular,
if $n$ is fixed and $d=o(p/\log^2 p),$  then
$$|f(I_i)\cap I_j|\sim p/n^2.$$
\end{theorem}

\noindent
This follows at once from the  more numerically precise statement  in Theorem \ref{smallgcd} below, and relies on bounds for
binomial exponential sums
\be \label{binsum}  \sum_{x=1}^{p-1} e_p(ax^k+bx). \ee

As we show in Theorem \ref{biggcd}
below, if we avoid those few cases in  Example \ref{ex1},  then  even for large $d$,  for a given $n$ there are at most finitely many cases of Type  (iii); that is for all other mappings the image of each residue class $f(I_i)$ hits at least two different residue classes mod $n$.

\begin{theorem}\label{main} If $n$ is even and $f(x)\neq \pm x$ mod $p$  or if $n$ is odd and  $f(x)\neq \pm x$ or $\pm x^{(p+1)/2}$ mod $p$,  then there are no $i$, $j$ with $f(I_i)\subseteq I_j$ once
$$p\geq  9\cdot 10^{34}\: n^{92/3}. $$
\end{theorem}

In the linear case we can be even more precise:

\begin{theorem} \label{mainiiik=1} Suppose that $f(x) = Ax$  mod $p$.

For $p>2n$ there are no Type (iii)  linear maps  $f(x)\neq \pm x$ mod $p$.

\end{theorem}

Similarly for the maps with  $k=(p+1)/2,$ but  not of the form  considered in Example \ref{ex1},
we can refine the bound in Theorem \ref{main}.
\comments{
\begin{theorem}\label{k=(p+1)/2}
Suppose that $f(x)\equiv Ax^{(p+1)/2}$ mod $p$ with $A\neq \pm 1$.

If $p>10^6$ and $p>4275 \:n^2 \log^4n$ then $f(x)$ is not a Type (iii) map.
\end{theorem}
}

\begin{theorem}\label{square}  Suppose that
$$ f(x)=Ax^{(p+1)/2} \text{ mod } p, \;\;\; A\neq \pm 1. $$
If $n\geq 2$ and $p>(4n+1)^2$ then $f(x)$ is not a Type (iii) map.
\end{theorem}
\noindent

Theorems \ref{mainiiik=1}, \ref{square} and Example \ref{ex1}, are the cases where the integer
 $$L:=(p-1)/d $$
is 1 or 2.
When $L\geq 3$ is small the argument in Theorem \ref{biggcd} similarly shows that there are
no Type (iii) maps $f(x)=Ax^k$ mod $p$ once
$$p> 1214 \: n^2 (L-1)^2\: \log^4(n(L-1)). $$

At the other extreme,  for the Lehmer type maps,  $k=-1$, we have $d=2$ and  \eqref{asymptotic} certainly  gives an asymptotic formula, but
using  the  Kloosterman sum bound $2\sqrt{p}$ for \eqref{binsum}  drastically  improves the error term (see for example \cite{Alkan}).
More generally for small $|k|$, $k\neq 1$,  one can use the Weil \cite{Weil} bound $|k-1|\sqrt{p}$ to obtain (see also \cite{Igor2})
$$  |f(I_i)\cap I_j| =\frac{p}{n^2} +O(|k|p^{1/2}\log^2p). $$
Similarly when $|k|$ is small we can obtain good bounds for both Type (iii) and (iv).

\begin{theorem} \label{lehmer} Suppose that $f(x) = Ax^{k}$  mod $p$ with $k\neq 1$ positive or negative.

(a) If $p> 37|k-1|^2 n^2$ then $f(x)$ is not a Type (iii) map.

(b) If $p\geq 16.2 |k-1|^2 n^4$ then $f(x)$ is not a Type (iv) map.

\end{theorem}

If  we want a stronger statement avoiding cases of  Type (iv) even when $d$ is large, that is, prove that the image of every residue class mod $n$  hits every residue class mod $n,$ then we will need to exclude more examples for $n>2$.  We explore this problem in \cite{GKII}.  The proofs of Theorems \ref{main}, \ref{mainiiik=1} and \ref{lehmer} are given in Section \ref{proofmain} and Theorem \ref{square} in  Section \ref{Squaresection}.

For a given $n$ we know from Theorem \ref{main} that there are at most finitely many occurrences of Type (i), (iia) and (iib),  but of course the bounds in this paper are far too large to
obtain a complete determination as was done for $n=2$ in \cite{CochKony}. We hope to employ the methods of \cite{CochKony} to complete this determination in a subsequent work.

\section{Computations and Conjectures}\label{computations}

Computations looking for maps of Type (i)-(iv) were performed for the primes $p<20,000$, exponents $k<p-1$ and  moduli $n=3$ through 12.  Of particular interest was obtaining Type (iii) examples with the ratio $p/n$ as large as possible. This led to a more extensive investigation of the exponent $k=(p+1)/2$. We quickly discovered Example \ref{ex1} where for any odd $n$ and prime $p \equiv 1 $ mod 4, the mapping $f(x)= \pm x^{(p+1)/2}$ mod $p$ is Type (iii).  Further families with this exponent are given in Theorems \ref{2xp2} to \ref{3xp2}. They were all discovered by looking at patterns in the data.  Notice that a Type (iii) map of the form $f(x)=\pm x^{(p+1)/2}$ mod $p$  must produce a Type (iib) map  for $f(x)$ or $-f(x)$; of course we are only interested  maps of this type  for $n$ or even, or for $n$ odd where the $f(I_i)=I_i$ has  $2i\not\equiv  p$ mod $n$.

\subsection{Type (iii) mappings:}
In   Theorem \ref{main} we verified the existence of a constant $K(n)$ such that for $p>K(n)$ and $f(x)\neq \pm x$ and   when $n$ is odd $f(x)\neq \pm x^{(p+1)/2}$ mod $p$, every residue class is mapped to at least two different  residue classes, that is, $f(x)$ is not Type (iii).  The constant  $K(n)= 9\cdot 10^{34} n^{92/3}$  obtained there is undoubtedly far from the truth. Table \ref{tablei} gives the five largest primes having an
$f(x)=Ax^k$ mod $p$ with  $f(I_i)\subseteq I_j$  for some $i,j$,  found for each $3\leq n\leq 12$ and $2n<p<20,000$.
Since  $Ax^k$  has this property if and only if  $-Ax^k$ does, we just consider  positive $A$. From this data we make the following conjecture.

\begin{conjecture} The optimal values for $K(n)$ for $n=3$ through $12$ are
\begin{align*} &  K(3)=  17, \;\; K(4)=13,\;\; K(5)=43,\;\; K(6)=17,\;\; K(7)=37,\\
 &  K(8)=43,\;\; K(9)=43,\;\; K(10)=47,\;\; K(11)=67,\;\; K(12)=53. \end{align*}
The data suggests that one can take $K(n)= 2n^2$, although the correct bound is likely of order  somewhere between $n \log n \log \log n$ and $n^2$.
\end{conjecture}

\begin{table}[tbp]

\parbox{.3\linewidth}{
\footnotesize\begin{tabular}{|cccc|}\hline
  $p$ & $A$ & $ k$ & $i$ \\\hline
& & $n=3$ &  \\
7 & 3 & 5 & 0,1,2* \\
 11 & 4  & 9 & 1* \\
 13 & 3 & 5,11 & 2*  \\
17 & 4 & 5,13 & 1*\\\hline
 & & $n=4$ & \\
11 & 1 & 9 & 0,1*,2*,3\\
13 & 2 & 5 & 0,1,2,3 \\\hline
 & & $n=5$ & \\
19 & 5 & 17 & 2*\\
23 & 10 & 21  & 4* \\
29 & 14 & 13 & 0,4 \\
31 & 1 & 11 & 3* \\
43 & 6 & 29 & 4*  \\\hline
 & & $n=6$  & \\
13 & 1 & 5,11 & 3*,4* \\
13 & 1 & 7 & 2,3,4,5 \\
13 & 3 & 5 & 2,5  \\
13& 3 & 11  & 2*,5* \\
13 & 6 & 11 & 0,1 \\
17 & 1 & 9 & 0,5 \\
17 & 2 & 5 & 2,3 \\
17 & 4 & 7 & 0,5 \\
17 & 4 & 15 & 0*,5* \\
17 & 8  & 13  &  1,4\\ \hline
 & & $n=7$ & \\
19 & 2 & 17 & 6* \\
19 & 3 & 7 & 0,5 \\
19 & 3 & 11 & 6* \\
19 & 3 & 17 & 0*,5* \\
19 & 5 & 5 & 6* \\
19 &  6 & 7,11& 0,5\\
19 &  7 & 7 & 6* \\
19 & 7 & 11 & 0,5\\
19 & 8 & 13 & 6* \\
23 & 8 & 21 & 1* \\
23 & 9 & 21 & 3*,6* \\
29 & 14 & 13,27 & 4* \\
31 & 2 & 29 & 5* \\
37 & 16 & 17 & 4*,5* \\ \hline
\end{tabular}
}
\hfill
\parbox{.34\linewidth}{
\footnotesize\begin{tabular}{|cccc|}\hline
$p$ & $A$ & $ k$ & $i$ \\\hline
 & & $n=8$ & \\
23 & 2 & 3 & 0,7 \\
23 & 3,10 & 5 & 0,7\\
23 & 6,11 &  17 & 0,7 \\
23 & 1 & 19 & 0,7\\
23 & 10 & 21 & 0,7 \\
29 & 1 & 15 &  2,3 \\
29 & 7& 19 & 6,7  \\
31 & 5 & 11 & 3*,4*  \\
41 & 1 & 21 & 3,6 \\
43 & 2 & 13 & 0,3\\\hline
 & & $n=9$ & \\
29 & 5 & 9,23 & 4,7\\
29 & 9 & 11,25 & 4,7\\
29 & 10 & 13,27 & 1*\\
31 & 9 & 29 & 2*\\
37 & 4 & 17,35 & 5*\\
41 & 1 & 19 & 1*,4*\\
41 & 3 & 11,31 & 7*\\
41 & 4 & 13,33 & 7*\\
41 & 10 & 9,29 & 7*\\
41 & 11 & 19,39 & 7*\\
41 & 12 & 3,23 & 7*\\
41 & 13 & 7,27 & 7*\\
41 & 18 & 17,37 & 7*\\
43 & 7 & 41 & 8*\\\hline
   & & $n=10$ & \\
31 & 1 &  11 &  3*,5*,6*,8*\\
37 & 8  & 7  & 8,9\\
37 &14 &31 &0,7\\
41 & 18 & 19& 2,9\\
41 &2 & 21  & 3,8\\
41 & 20 & 21 &  5,6\\
43 &6 &29 & 4*,9*\\
47  &11 & 17 & 8,9\\ \hline
\end{tabular}
\vspace{14ex}
}
\hfill
\parbox{.28\linewidth}{
\footnotesize\begin{tabular}{|cccc|}\hline
  $p$ & $A$ & $ k$ & $i$ \\\hline
 &  & $n=11$ & \\
41 & 6  &3,23  & 4*\\
41 & 17 & 7,27  & 4*\\
41  & 16 & 9,29 & 4*\\
41 & 14 & 11,31 &4*\\
41 & 18 &13,33 &4*\\
41 & 10 &17,37 &4*\\
41 & 19  &19,39 &4*\\
41 & 18 & 29 & 1,7\\
43 & 18 &23 &0,10\\
43 & 6  & 41 &5*\\
47 &15 & 45& 7*\\
53 &2 & 25 &10*\\
53 &2  &51 &10\\
67 &29 &23 &6*\\ \hline
 & & $n=12$ & \\
31 &5 & 7 & 0,7\\
31 &7,14 &7 &8,11\\
31 &5& 11 &0*,7*\\
31 &6& 11 & 1,6,9,10\\
31 &10& 13& 0,7\\
31 &6 & 17 & 0,7\\
31 &8 &19 &0*,7*\\
31 &15 & 19& 9*,10*\\
31 &4,8 &23 & 8,11\\
31 &12& 23& 0,7\\
31 &3 &29 &9*,10*\\
31 &5 &29 &8*,11*\\
31 &9 & 29 &0*,7*\\
37 &1 &19 &4,5,8,9\\
41 &9 &3,13,23,33 &8,9\\
41 &1 & 11,21& 8,9\\
41 &20& 19,29 & 7,10\\
41 &1 &31 &7*,8,9,10*\\
43 &12 & 37 & 9,10\\
53 &1 &  27 &1,4\\ \hline
\end{tabular}

\vspace{11.5ex}}

\vspace{2ex}

\caption{Type (iii): Five  largest  primes $2n<p<20,000$  having an $f(x)=Ax^k$ mod $p$ with $f(I_i)\subseteq I_j$  for some $i,j$
 ($f(x)\neq x$ if $n$ is even, and  $f(x)\neq x$ or $x^{(p+1)/2}$ if  $n$ is odd). \newline
Type  (iib): Cases of  $f(I_i)=I_i$,   are marked with a *.}
\label{tablei}
\end{table}

Looking for larger ratios of $p$ to $n,$ we extended our computations to $13\leq n \leq 86$ and  $5n\leq p\leq 15n$. The  values 
found with  $p/n>9$ are recorded in Table \ref{tablento70}. 

\begin{table}[tbp]

\footnotesize\begin{tabular}{|cccccc|}\hline
$n$  & $p$ & $A$ & $k$  & $i$ &  $p/n$ \\ \hline
70 & 641 & 1 & 321 & 27,54 &  $9.157142\ldots$\\
84 & 773 & 1 & 387 & 27,37,64,74 & $9.202380\ldots$\\
30 & 277 & 1 & 139 & 10,27 & $9.233333\ldots$\\
39 & 367 & 1 & 245 & 8 & $9.410256\ldots$\\
62 & 593 & 1 & 297 & 16,19 & $9.564516\ldots$\\
82 & 809 & 1 & 405 & 20,51 & $9.865853\ldots$\\
60 & 593 & 1 & 297 & 21,32 & $9.883333\ldots$\\
37 & 367 & 84 & 245 & 17 & $9.918918\ldots$\\
85 & 853 & 221 & 143,569 & 44 & $10.035294\ldots$\\
83 & 853 & 220 & 143,569 & 53 & $10.277108\ldots$\\
35 & 367 & 83 & 245 & 26 & $10.485714\ldots$\\
81 & 853 & 220 & 143,569 & 62 & $10.530864\ldots$\\
76 & 809 & 1 & 405 & 58,67 & $10.644736\ldots$\\
79 & 853 & 221 & 143,569 & 71 & $10.797468\ldots$\\
86  & 941 &  1 & 471 &  83,84 &  $10.941860\ldots$\\
84 & 977 & 1 & 489 & 12,41 & $11.630952\ldots$\\ 
86  & 1013 & 1  & 507 & 2,65 &  $11.779069\ldots$ \\ \hline
\comments{30 & 277 & 1 & 139 & 27 & $9.2333$\\
39 & 367 & 1 & 245 & 8 & $9.4103$\\
62 & 593 & 1 & 297 & 16,19 & $9.5645$\\
82 & 809 & 1 & 405 & 20,51 & $9.8659$\\
60 & 593 & 1 & 297 & 21,32 & $9.8833$\\
37 & 367 & 84 & 245 & 17 & $9.9189$\\
83 & 853 & 220 & 143, 569 & 53 & $10.2771$\\
35 & 367 & 83 & 245 & 26 & $10.4857$\\
81 & 853 & 220 & 143, 569 & 62 & $10.5309$\\
76 & 809 & 1 & 405 & 58, 67 & $10.6447$\\
79 & 853 & 221 & 143, 569 & 71 & $10.7975$\\
84 & 977 & 1 & 489 & 12, 41 & $11.6310$\\  \hline
52 & 457 & 1 &  229  & 12 &  8.78846153846154\\
52 & 457 &  1 &  229 & 29  & 8.78846153846154\\
59 & 521 &  5 & 209  & 54 &  8.83050847457627\\
59 & 521 &5 &  469 & 54 & 8.83050847457627\\
64 &569 & 1  &285 & 19 & 8.890625\\
64 & 569 &1 & 285 & 38& 8.890625\\
64 & 569 & 1  & 285  & 49 &8.890625\\
64  & 569 & 1& 285 &8  & 8.890625\\
41 &367 & 1& 245 & 40 & 8.95121951219512\\
70 &  641 & 1 &321 &27  & 9.15714285714286\\
70 &  641 & 1 & 321 & 54  & 9.15714285714286\\
30 &  277 &  1&139  & 10  & 9.23333333333333\\
30 &277 & 1 &139 & 27  & 9.23333333333333\\
39 & 367 & 1 & 245  & 8  &9.41025641025641\\
62 & 593 & 1  &297 & 16 &9.56451612903226\\
62 &593 &1 &297 &19 &9.56451612903226\\
60 &593 &1 & 297 &21 & 9.88333333333333\\
60 &593 &1 & 297 &32 &9.88333333333333\\
37 &367 &84 & 245 &17 & 9.91891891891892\\
35 &367 &83 & 245 &26 & 10.4857142857143\\ \hline}
\end{tabular}
\vspace{1ex}
\caption{Type (iii) with   $3\leq n\leq 86$  and $9< p/n<15$\newline
($f(x)=Ax^k$ mod $p$ with $A>0$, and  $f(x)\neq x^{(p+1)/2}$  if  $n$ is odd).}
\label{tablento70}
\end{table}

A large number of the Type (iii) maps with $p/n$ large have $k=(p+1)/2$.
 The bounds in Example \ref{ex1} and Theorem \ref{square} enable 
a complete determination when  $k=(p+1)/2$ and $p>2n$ for small $n$. There are no such Type (iii) mappings for $n=3,4,7$ or $9,$ 
with  a complete list of such  maps for the remaining $5\leq n\leq 12$ shown in Table \ref{tablek=(p+1)/2smalln}.

\begin{table}[tbp]

\parbox{.323\linewidth}{
\footnotesize\begin{tabular}{|cccc|}\hline
 & $p$ & $A$ & $i$  \\ \hline
 $n=5$ & 13 & 4,5 & 3,5    \\\hline
 $n=6$ & 13 &  1 & 2,3,4,5  \\
    & 17 & 1 & 5,6 \\ \hline
$n=8$ & 17 & 1 & 1,3,6,8 \\
 & 29 & 1 & 2,3 \\
 & 41 & 1 & 3,6 \\ \hline
\end{tabular}}
\hfill
\parbox{.325\linewidth}{
\footnotesize\begin{tabular}{|cccc|}\hline
 & $p$ & $A$ & $i$  \\ \hline
$n=10$ & 29 & 2 & 2,7 \\
  & 29 & 9,10 & 9,10\\
  & 29 & 14 & 4,5 \\
 & 41 & 2 &  3,8\\
  & 41 & 20 & 5,6 \\\hline
\end{tabular}
\vspace{2ex}}
\hfill
\parbox{.323\linewidth}{
\footnotesize\begin{tabular}{|cccc|}\hline
 & $p$ & $A$ & $i$  \\ \hline
 $n=11$ & 29 & 10 & 7,11 \\\hline
$n=12$ & 29 & 1 & 1,2,3,4 \\
   & 29 & 12 & 8,9 \\
   & 37 & 1 & 4,5,8,9 \\
 & 41 & 1 & 8,9 \\
  & 53 & 1 & 1,4 \\\hline
\end{tabular}}

\vspace{1ex}
\caption{All  Type (iii) of the form $f(x)=Ax^{(p+1)/2}$ mod $p$ for $3\leq n\leq 12$
 (for $A>0$ and excluding $f(x)=x^{(p+1)/2}$ if  $n$ is odd).}
\label{tablek=(p+1)/2smalln}
\end{table}

In the proof of Theorem \ref{square} for $k=(p+1)/2$ we had to deal separately with  the case $A=2$, so additional computations were performed
for $f(x)=2x^{(p+1)/2}$ mod $p$ looking for examples with large ratio $p/n$.  These corresponded to primes with a certain 
pattern of quadratic residues.  Examining the corresponding $n$ values led us to a family 
of Type (iii) mappings of the form $2x^{(p+1)/2}$, with arbitrarily large $p/n$, and requiring $K(n)$ to be as large as $n\log n$. 

\begin{theorem} \label{2xp2} Let
$$   f(x)=2x^{(p+1)/2} \text{ mod } p. $$
Suppose that $p\equiv 1$ mod $4$  has
\be \label{A=2property}  \left(\frac{p}{q} \right) =\begin{cases}  +1, & \text{ if $q=1$ mod $4$,} \\  -1, & \text{ if $q=3$ mod $4$,} \\      \end{cases}\ee
for all primes $3\leq q \leq 4t-1$,
 and that $n\equiv 2$ mod $4$ with
\be \label{nrange} \frac{2p}{4t+1} \leq n < \frac{2p}{4t-1}.\ee
Then  for both
\be \label{defi}  i:=\frac{1}{4}(2p-(4t-1)n ), \;\;\;\;\;  j:=\begin{cases} 2i \text{ mod } n, & \text{ if $\left(\frac{n}{p}\right)=-1$,}\\ p-2i \text{ mod } n, & \text{ if $\left(\frac{n}{p}\right)=1$,}
\end{cases} \ee
and
\be \label{defi2}  i:=\frac{1}{4}(2p-(4t-3)n ), \;\;\;\;\;  j:=\begin{cases} p-2i \text{ mod } n, & \text{ if $\left(\frac{n}{p}\right)=-1$,}\\ 2i \text{ mod } n, & \text{ if $\left(\frac{n}{p}\right)=1$,} \end{cases}
\ee we have $f(I_i)\subseteq I_j$.
\end{theorem}

The primes $p<100,000$ with  property \eqref{A=2property} with $t\geq 9$, and the smallest $n$ this gives in Theorem \ref{2xp2}  are shown in Table \ref{TableA=2}
\begin{table}[tbp]
\small\begin{tabular}{|ccccc|}\hline
 &  $p$  & $t$ & $n$ & $ p/n$ \\\hline
&15461 & 9  & 838 &  18.449880\ldots \\
&23201  & 9  & 1258 & 18.442766\ldots  \\
Theorem \ref{2xp2} & 40169 & 9 &  2174 &  18.477000\ldots \\
&70769 & 10  &  3454  & 20.488998\ldots\\
 &75869 & 9 &  4102  & 18.495611\ldots  \\ \hline
\end{tabular}

\vspace{2ex}
\caption{Primes $p<100,000$ with $p\equiv 1$ mod $4$ and $\left(\frac{-p}{q}\right)=1$ for all odd $q\leq 4t-1$ for some  $t\geq 9$.}
\label{TableA=2}

\end{table}

 Using the Chinese remainder we can construct $p$ with this property for arbitrarily large $t$. For example we could  take $p\equiv 1$ mod $4Q_1$ and $-1$ mod $Q_2$ where $Q_1$ and $Q_2$ are the products of the primes $q\leq 4t-1$ that are $1$ or $-1$ mod $4$ respectively (there are of course many other ways).  Hence we can make Type (iii)
examples with $p>(2t+\frac{1}{2}-\ve)n$. In particular we can't take $K(n)=Cn$ however large the $C$. Moreover, by the work of
Heath-Brown \cite{hb} and Xylouris \cite{xylouris} on the smallest prime in an arithmetic progression, there exist such $p$ with $p\ll  Q^{5.18}$, with $Q=4Q_1Q_2$,
and hence examples of Type (iii) with $p> \frac 1{11} n \log n$.  Assuming GRH guarantees such $p<2(Q\log Q)^2$  (see Bach \cite{Bach} or Lamzouri, Li and Soundararajan \cite{Xiannan}) and thus $p> (\frac{1}{4}-\ve) n\log n$.
The proofs of the theorems in this section are given in Section \ref{SpecialTypeiii}.

\subsection{Type (iib) Mappings}
The Type (iib) maps,  where $f(I_i)=I_i$  for some $i$, are marked with an asterisk in Table \ref{tablei}; of course for such cases
the iterates will also fix $I_i$
and a number of these can be seen in the table. For example for $n=7$, $p=19$, the map  $f(x)=5x^5$ mod $19$ fixes $I_6$, as does
$f^2(x)=7x^7$ mod $19$, $f^3(x)=-2x^{17}$ mod $19$,  $f^4(x)=-8x^{13}$ mod $19$,  $f^5(x)=-3x^{11}$  mod $19$
and $f^6(x)=x$ mod $p$; in this case $p-6\equiv 6$ mod $n$ so that the maps with negative $A$
recorded in their positive guise also fix  $I_6$. 

Many examples of Type (iib) mappings in our data with large ratio $p/n$ are of the form $f(x)=\pm x^{(p+1)/2}\equiv \pm \left(\frac{x}{p}\right)x$ mod $p$. For $n$ even, or $n$ odd with $i\not\equiv 2^{-1}p$ mod $n$, it is readily seen that this requires  $p$ to have a   string of  roughly $p/n$ consecutive quadratic residues or nonresidues. In Theorems \ref{1xp2} to \ref{3xp2} we explore how   conversely  long blocks of  consecutive residues or nonresidues can produce large $p/n$ values. We distinguish several cases frequently encountered in the data.
Theorem \ref{1xp2} deals with consecutive quadratic residues starting at 1, Theorem \ref{central} with an interval of consecutive residues or nonresidues around $p/2$, Theorem \ref{p/3} with intervals around $p/3$ and $2p/3$,
and Theorem \ref{3xp2} with the remaining cases. Table \ref{tablebigp/n} shows  the primes $p<100,000$ with a string of at least 25 consecutive residues or nonresidues, and 
examples arising from them when Theorems \ref{1xp2} through \ref{3xp2} 
are applied as appropriate. For the non-central interval we give the $[a,a+t)$ with $a<p/2$, and omit  the symmetric interval  $(p-a-t,p-a]$.

\begin{theorem} \label{1xp2} Let $t$ be a positive integer and $p>t$ a prime with $p \equiv 1$ mod $8$ and $\left(\frac{p}{q}\right)=1$ for all odd primes $q\leq t$. Then for any
$n> (p-1)/(t+1)$ the map $f(x)=\left(\frac{n}{p}\right)x^{(p+1)/2}$ mod $p$ is the identity map on $I_0$.
\end{theorem}

Notice that $i=0$ does not have $2i\equiv p$ mod $n$, so these examples are of interest for both odd and even $n$.
Again, by the Chinese Remainder Theorem and Dirichlet's theorem,  for any $t$, there exist  infinitely many $p$ satisfying the hypotheses of this theorem and so we get examples with $p$ as large as $n\log n$, but
for certain $p$ we can push the size of $K(n)$ a little bigger.
By the work of Graham and Ringrose  \cite{Ringrose} we know there exist infinitely many primes $p \equiv 1$ mod $4$ having a least quadratic nonresidue of size at least $c \log p \log \log \log p$ for some constant $c$ (with improvement to $c \log p \log \log p$ under GRH by Montgomery \cite{Montgomery}). Taking $t=\lfloor c \log p \log \log \log p\rfloor$, the hypotheses of the theorem are satisfied by reciprocity, and thus with $n=\lceil (p-1)/t\rceil$, we obtain a Type (iib) mapping with $p\gg n \log n \log \log \log n$ (with improvement under GRH).

Since $p\equiv 1$ mod $4,$ if  our interval of consecutive quadratic  residues or nonresidues contains $(p-1)/2,$ then we have a symmetric 
interval around $p/2$ with $\left(\frac{x}{p}\right)=\left( \frac{(p-1)/2}{p}\right)=\left(\frac{2}{p}\right)$. For odd $n$ we 
obtain examples with $p/n$ close to the interval length $t$, but  unfortunately \eqref{nodd} has $2i\equiv p$ mod $n$ which we know always gives a Type (iii) by Example \ref{ex1}, though additionally here  $f(x)$ is  the identity map on $I_i$. If we restrict to even $n$ then the ratio $p/n$ is only close to $t/2$.

\begin{theorem}\label{central}
Supppose that $p\equiv 1$ mod $4$ has $t=2T$ consecutive residues or nonresidues around $p/2$:
$$ \left( \frac{x}{p}\right) =\left(\frac{2}{p}\right),\;\;\;\;\; a=\frac{p+1}{2}-T \leq x \leq \frac{p-1}{2}+T. $$
Equivalently suppose that $\left(\frac{q}{p}\right)=1$ for all odd primes $q\leq 2T-1$. 

Suppose $n$ is even with
\be \label{neven}  \frac{2p}{t+1} < n < \frac{2p}{t-1}, \;\;\; i:=an -\left(\frac{n}{2}-1\right)p,\ee
or $n$ is odd with
\be \label{nodd}  \frac{p}{t+1} < n < \frac{p}{t-1}, \;\;\; i:=an -\left(\frac{n-1}{2}\right)p,\ee
then $f(x)= \left(\frac{2n}{p}\right)x^{(p+1)/2} $ mod $p$ is the identity map on $I_i$.

\end{theorem}

\begin{table}[tbp]

\footnotesize\begin{tabular}{|c|cccccc|}\hline
  &    $p$ &  $t$ & $a$ & $n$ & $i$ & $p/n$ \\ \hline
Theorem \ref{1xp2} & 87481 & 28 & 1  &  3017 & 0 &  29.115346\;\;\; \\
  ($n$ odd \& even) & 87481 & 28 & 1  &  3018 &  0 &  29.105699\;\;\; \\\hline
  & 13381 & 28 & 6677 & 463  & 440  & 28.900647\ldots  \\
& 20749 & 28 & 10361 & 717 & 695  & 28.938633\ldots \\
&  51349 & 28 & 25661  & 1771 & 1766  & 28.994353\ldots  \\
 Theorem \ref{central} &82021 & 30 & 40996  & 2647  & 2629  & 30.986399\ldots  \\
 ($n$ odd) &87481 & 28 & 43727 & 3017 & 3011  & 28.996022\ldots  \\
&89989 & 28 & 44981 & 3105 & 3077  &28.981964\ldots  \\
&92821 & 28 & 46397 & 3201 & 3197  & 28.997500\ldots \\
&99709 & 30 & 49840  & 3217 & 3208  & 30.994404\ldots \\\hline
  & 13381 & 28 & 6677 & 924  & 907  & 14.481601\ldots  \\
& 20749 & 28 & 10361 & 1432 &1417  & 14.489525\ldots \\
&  51349 & 28 & 25661  & 3542&  3532 &    14.497176\ldots  \\
 Theorem \ref{central} &82021 & 30 & 40996  & 5292  & 5287  & 15.499055\ldots  \\
 ($n$ even)  &87481 & 28 & 43727 & 6034 & 6022  & 14.498011\ldots  \\
&89989 & 28 & 44981 & 6208 & 6181 &   14.495650\ldots  \\
&92821 & 28 & 46397 &  6402 &6394  & 14.498750\ldots  \\
&99709 & 30 & 49840  & 6434 & 6416 & 15.497202\ldots \\\hline
\end{tabular}

\vspace{2.5ex}

\footnotesize\begin{tabular}{|c|cccccccc|}\hline
 Theorem \ref{p/3}     & $p$ & $t$ & $a$ &  $T_1$ & $T_2$ & $n$ & $i$ & $p/n$ \\\hline 
$n\equiv 2$ mod $3$   & 52361 & 29 &  17437 & 17 & 12 & 1976 & 1974 & 26.498481\ldots \\
   &  65129 &  27 & 21693 & 17 & 10 & 2459 & 2436 &  26.485969\ldots \\\hline
 $n\equiv 1$ mod $3$   & 52361 & 29 &  17437 & 17 & 12 & 2833 & 2800 &  18.577832\ldots \\
 &  65129 &  27 & 21693 & 17 & 10 & 4204 & 4182 &  15.492150\ldots \\\hline
$n\equiv 0$ mod $3$   & 52361 & 29 &  17437 & 17 & 12 & 2964 & 2961 & 17.665654\ldots \\
   &  65129 &  27 & 21693 & 17 & 10 & 3696 & 3529 &  17.621482\ldots \\\hline
 \end{tabular}

\vspace{2.5ex}
\begin{tabular}{|c|ccccccc|}\hline
  &    $p$ &  $t$ & $a$  & $u$ & $n$ & $i$  & $p/n$ \\ \hline
Theorem \ref{3xp2}  &90313 & 26 & 39556  & 38 & 3386 & 2437,3226 &  26.672474\ldots  \\\hline
\end{tabular}
\vspace{2ex}
\caption{Type (iii): Primes $p<100,000$  with $t\geq 25$ consecutive quadratic residues or consecutive nonresidues, $[a,a+t)$.}
\label{tablebigp/n}
\end{table}

A large interval of consecutive quadratic residues or nonresidues around $p/3$ (and hence under $x\mapsto p-x$ around $2p/3$)
will also lead to large $p/n$ values, the size depending on $n$ mod $3$.

\begin{theorem} \label{p/3} Suppose that $p\equiv 1$ mod $4$ and set 
$$ \delta := \begin{cases} 1 & \text{ if $p\equiv 1$ mod $3$, }\\2 & \text{ if $p\equiv 2$ mod $3$. }\end{cases} $$
Suppose that
$$ \left(\frac{x}{p}\right) = \left(\frac{3}{p}\right), \;\;\; $$
for
\begin{align*} a_1: & =\frac{1}{3}(p-\delta) -(T_1-1)\leq x \leq \frac{1}{3}(p-\delta) +T_2 ,\\
\;\;\; a_2:& =\frac{1}{3}(2p+\delta) -T_2\leq x \leq \frac{1}{3}(2p+\delta) +(T_1-1).\end{align*}
Equivalently 
\be  \label{equiv} \left(\frac{3m-\delta}{p}\right)=1,\;\; 1\leq m\leq T_2,\;\;\;\;\; \left(\frac{3m+\delta}{p}\right)=1,\;\; 0\leq m <T_1. \ee

Suppose that $n\equiv 0$ mod $3$ has 
$$ \frac{3p}{3T_1+\delta} < n < \frac{3p}{3T_1+\delta -3},\;\;\; \;\;i:= a_1n-\left(\frac{n}{3}-1\right)p, $$
or
$$  \frac{3p}{3T_2 +3-\delta} < n < \frac{3p}{3T_2-\delta},\;\;\;\;\; i:= a_2n-\left(\frac{2n}{3}-1\right)p, $$
or $n\equiv 2$ mod $3$,   when  $T_1 \leq 2T_2+2-\delta$  and
$$  \frac{2p}{3T_1 +\delta} < n < \frac{2p}{3T_1+\delta-3},\;\;\;\;\; i:= a_1n-\left(\frac{n-2}{3}\right)p, $$
or $n\equiv 2$ mod $3$,   when  $T_1 \geq 2T_2+2-\delta$  and
$$  \frac{p}{3T_2+3 -\delta} < n < \frac{p}{3T_2-\delta},\;\;\;\;\; i:= a_2n-\left(\frac{2n-1}{3}\right)p, $$
or $n\equiv 1$ mod $3$,   when  $T_2 \geq 2T_1-1+\delta$  and
$$  \frac{p}{3T_1 +\delta} < n < \frac{p}{3T_1+\delta-3},\;\;\;\;\; i:= a_1n-\left(\frac{n-1}{3}\right)p, $$
or $n\equiv 1$ mod $3$,   when  $T_2\leq 2T_1 -1+\delta$  and
$$  \frac{2p}{3T_2+3 -\delta} < n < \frac{2p}{3T_2-\delta},\;\;\;\;\; i:= a_2n-\left(\frac{2n-2}{3}\right)p. $$
Then $f(x)=\left(\frac{3n}{p}\right)x^{(p+1)/2} $ mod $p$ is the identity map on $I_i$.

\end{theorem}
Note if $p\equiv 1$ mod 3 then $\left(\frac{3}{p}\right)=1$ and a large interval around $p/3$ leads to a long interval starting at $a=1$ where one can use Theorem 2.2 to  potentially produce a larger $p/n$.
Similar theorems could no doubt be obtained for intervals around $p/5$ etc. Three  $p/3$ type examples occur in Table 2; namely 
$p=277$, $T_1=10$, $T_2=0$, where the largest $p/n$ in Theorem \ref{p/3}  will be for  $n\equiv 0$ mod $3$, with the smallest $n=27$
(smallest even $n=30$), $p=569$, $T_1=3$,$T_2=6$ where the best choice is $n\equiv 1$ mod 3, smallest $n=61$ (smallest even $n=64$), and $p=641$, $T_1=5$, $T_2=6$ where the smallest $n\equiv 1$ mod $3$ is $n=70$.

For general intervals of consecutive quadratic residues or nonresidues we have:
\begin{theorem}\label{3xp2}
Suppose that  we have $t$ consecutive quadratic residues or nonresidues  mod $p$ starting at an $a\geq 2$
$$\left(\frac{x}{p}\right)=\left(\frac{a}{p}\right),\;\; a\leq x\leq a+t-1,\;\;\; a=st+r,\;\;0\leq r<t.$$
If $ i=na-(s-u)p$ and $n$ is an integer in
\be \label{nrange}\max\left\{\frac{(s-u+1)p}{a+t}, \; \frac{(s-u)p}{a}\right\}  \leq n \leq  \frac{(s-u)p}{a-1}, \ee
or  $ i=(s-u+1)p-n(a+t-1)$ and
\be \label{nrange2}\frac{(s-u+1)p}{a+t} \leq n \leq  \min\left\{ \frac{s-u}{a-1}, \; \frac{s-u+1}{a+t-1}\right\} p, \ee
for some integer $\frac{(s+1-r)}{(t+1)} > u > s+1-\frac{1}{2}(a+t)$, then 
\be \label{deff} f(x)= \left(\frac{an}{p}\right) x^{(p+1)/2} \text{ mod } p, \ee
is the identity map on $I_i$.

\end{theorem}

Notice that for very large $a+t$  and a given $u$ we are not guaranteed an $n$ in the range \eqref{nrange}, but for small $a$, for example $a+t<\sqrt{p},$ there will be at least one $n,$ leading to an example with $p/n$ close to $t$ for small $u$.
Computationally searching the primes $p<10,000$ for Type (iii) mappings of the form $f(x)=x^{(p+1)/2}$ mod $p,$ with
$2i\not\equiv p$ mod $n$  if $n$ is odd, found $578$
primes $p\equiv 1$ mod $4$ admitting an $n$ with $7<p/n<20$. In each case we checked the minimal $n$ found against 
Theorem \ref{3xp2}.
Of these $545$ corresponded to taking  the longest string of consecutive quadratic residues or nonresidues in 
Theorem \ref{3xp2} (for $a$,$t$ and a suitable $u$), with $16$ to taking the second longest, and the remaining $17$ cases found with shorter intervals.

Theorems \ref{1xp2} to \ref{3xp2} used runs of $x$ where  $x^{(p-1)/2}$ mod $p$ is constant. One could similarly 
consider intervals where $x^{(p-1)/3}$ mod $p$ is constant. We state the counterpart of Theorem \ref{central}, 
since several examples of an interval around $p/2$ appear in Table \ref{tablento70} (namely $p=853$ and $367$ with $t=10$ and $a=422$ or $179$ respectively). 

\begin{theorem}\label{thirds}
Suppose that  $p\equiv 1$ mod $3$ and that
\be \label{cubicruns}  x^{(p-1)/3}\equiv a^{(p-1)/3} \text{ mod } p,\;\;\;\;\; a=\frac{p+1}{2}-T \leq x \leq \frac{p-1}{2}+T. \ee
Then, with $t=2T$ and $n$ and $i$ as in \eqref{neven} and \eqref{nodd},  any  map
\be \label{f/3} f(x)=(an)^{2(k-1)}  x^{k} \text{ mod } p,\;\;k=j(p-1)/3+1,\;  j=1,2, \;\;\gcd(k,p-1)=1, \ee
will be the identity map on $I_i$, and when $n$ is odd any map
\be \label{f/6}  f(x)=(an)^{2(k-1)}  x^{k} \text{ mod } p,\;\;k=j(p-1)/6+1,\; j=1,5, \;\;\gcd(k,p-1)=1, \ee
will have $f(I_i)=I_i$.

\end{theorem}
This could be further generalized to intervals where $x^{(p-1)/q}$  mod $p$ is constant.

\subsection{Type (i) and (iia) Mappings}

Only a few cases were found where $Ax^k$ mod $p$  permutes every residue class:

\begin{example}  The only cases of Type  (i), that is $f(I_j)=I_j$ for all $j$, found  for $3\leq n\leq 12$ and $p<20,000$, were
$n=3$,  $(p;A,k)=(5; -1,3)$ and $(7;-3,5)$.

\end{example}

The examples of Type (iia) found, that is  where $f(I_1),\ldots ,f(I_n)$ is a
 permutation of $I_1,\ldots ,I_n,$ are shown in Table \ref{Table2a}. By symmetry we include only   $A>0$ and of course exclude $f(x)=x$.

\begin{table}[tbp]

\small\begin{tabular}{|ccccc|}\hline
    & $p$ & $A$ & $ k$ & $\sigma$ \\\hline
$n = 3$ & 5 &  1 & 3 & (02) \\
 & 7  & 3 &  5 &  (13) \\ \hline
 $n = 4$
 & 7 &  2 &  5 & (12)\\
& 11 &  1 &  9  & (03)\\
 & 13  & 2 &  5 & (0312) \\ \hline
 $n = 5$
& 7 &  1 &  5 &(03) (24) \\ \hline
 $n = 8$
& 11 &  2 &  9  & (03)(12)(46)\\
& 13 & 4  &5  & (06)(14)(23)(57)\\\hline
 $n = 9$
 & 11 & 1  & 3  & (0354)(2867) \\
 & 11 & 1 &  7 & (0453)(2768)\\
&  11  & 1 &  9 & (05)(26)(34)(78)\\
 & 13 &  1  &  7 & (58)(67)\\ \hline
$n = 10$ & 13 & 2 &11 & (08)(12)(35)(47)(69) \\\hline
$n = 11$
 & 13 &  1 &  5 & (07)(26)(39)(4\hspace{.5ex}10)\\
&13  &1  & 7 &  (02)(58)(67)\\
& 13 &  1 & 11 &  (06)(27)(39)(4\:10)(58)\\ \hline
\end{tabular}

\vspace{2ex}
\caption{Type (iia): $f(x)=Ax^k$ mod $p$ with  $f(I_i)= I_{\sigma(i)}$  for some permutation $\sigma,$ for $3\leq n\leq 12$, $n+1<p<20,000$ (with $A>0$ and $f(x)\neq x$).}
\label{Table2a}
\end{table}

The reoccurrence of $p=n+2$ and $p=n+3$ is easily explained; in these cases we have only one or two residue classes with two entries,
$I_1=\{1,-1\}$ fixed by any $f(x)=x^k,$ or $I_1=\{1,-2\},I_2=\{2,-1\}$  interchanged by
$f(x)=2x^{p-2},$ with the remaining singleton sets permuted. This leaves only a few examples with $p\geq n+4$.  We searched for further Type (iia)  examples with $n+4\leq p < 5n$  for $13\leq n\leq 100$. The results are shown in Table \ref{Table2aExtra}.  It turns out that all of these primes have the following property:

\begin{table}[tbp]

\small\begin{tabular}{|cccc|}\hline
  $n$  & $p$ & $A$ & $ k$ \\\hline
24 & 29  & 1 &  15\\
33 & 37  & 1  & 19 \\
35 & 41 &  1 &  21\\
48 & 53 &  1  & 27\\
 57 & 61 &  1  & 31 \\
 68 & 73 &  1 &  37\\
  69 & 73 & 1  & 37\\
 83 & 89 &  1 & 45 \\
 91& 101 & 1  &  51\\
 92 & 97 &  1 &  25,49,73\\
93 &  97 &  1 &  49\\
96 & 101 & 1 &  51 \\\hline
\end{tabular}

\vspace{2ex}
\caption{Type (iia): $f(x)\neq x$ for $13\leq n\leq 100$ and  $n+4\leq p<5n$.}
\label{Table2aExtra}

\end{table}

\begin{theorem}\label{legendre} If $p=n+w$ with $2 \le w <n$, $p\equiv 1$ mod $4$ and
$$\left( \frac{y}{p}\right)=\left( \frac{w-y}{p}\right), \hbox{ for all $1\leq y< w/2$}, $$
then $f(x)=x^{(p+1)/2}$  mod $p$ produces a Type (iia) permutation.

\end{theorem}

The additional $f(x)=x^{(p+3)/4}$ and $x^{(3p+1)/4}$ mod $p$  for $p=97$ are also predictable; for example
for $p\equiv 1$ mod 24 all three $f(x)$  will occur for $p=n+4$ or $p=n+5$ when
 $3^{(p-1)/4}\equiv 1$ mod $p$ or  $2^{(p-1)/4}\equiv 3^{(p-1)/4}$ mod $p$  respectively  (where one might expect
either of these  to occur roughly half the time, 193 and 97 respectively being the first cases of these), with similar conditions for $p=n+6$, $n+7$ etc that should hold  for a positive proportion of the time.
The sparsity of Type (i) and (iia) examples suggests the following conjecture.

\begin{conjecture} Suppose that $p>2n$.

Excluding $f(x)=x$ mod $p$, there are only finitely many examples of  Type (i), that is   $f(I_i)=I_i$ for all $i$.

Excluding $f(x)=\pm x$ mod $p$  there are only finitely many examples of  Type (iia), that is $f(I_i)=I_{\sigma(i)}$ for some permutation $\sigma$ of $\{0,1,\ldots, n-1\}$.

\end{conjecture}

Indeed there may be no Type (i) or (iia) maps  with $p>2n$ other than the few cases found above for $n=3$ or $4$.
From the Graham \& Ringrose  \cite{Ringrose} bound $g(p)\gg \log p \log\log\log p$ on the least quadratic nonresidue we see that we cannot replace the $p>2n$ in this conjecture by $p>n+C \log n,$ however large the $C$.


\section{Type $(iii)$ and type $(iv)$  intersections for small $d$}
 For
$j=0, 1,\ldots , n-1$, let $I_j$ be the set of values in \eqref{Ij}.  Put
\begin{equation} \label{Nj}
N_j:=|I_j|= \begin{cases} \left\lfloor \frac {p-1+n-j}n\right\rfloor, & \text{if $j \neq 0$};\\
\lfloor \frac pn\rfloor,& \text {if $j=0$}.\end{cases}
\end{equation}

\begin{theorem}\label{smallgcd}   Suppose that $p>607.$  Then for any $A$ and $k$ satisfying  \eqref{range}, $2\leq n<p$, and $0\leq i,j\leq n-1,$ we have
\begin{equation} \label{smalldasym} \left| f(I_i)\cap I_j\right| = p^{-1}|I_i||I_j|  + E, \end{equation}
with
$$ |E|\leq  \left(d+1+2.293p^{89/92}\right)\left(\frac{4}{\pi^2}\log p + 0.381\right)^2, $$
and
$$ \frac{p}{n^2}-1 < p^{-1}|I_i||I_j|   < \frac{p}{n^2}+1. $$

 For $7\le p\le 607$, the same result holds with $.381$ replaced by $1/2$.
\end{theorem}

\begin{proof}

For any $i,j \in \{0,1, \dots, n-1\}$  write
$$ N_{ij}=|f(I_i)\cap I_j|. $$
 We use $\mathbb Z_p$ to denote the integers mod $p$, and view $I_i,I_j$ as subsets of $\mathbb Z_p$. We write $\mathscr{I}_i(x)$, $\mathscr{I}_j(x)$ for the characteristic functions for $I_i$, $I_j$ so that
$$ N_{ij}=\sum_{x \text{ mod } p} \mathscr{I}_i(x) \mathscr{I}_j(Ax^k). $$
Since $\mathscr{I}_j(x)$ is a periodic function mod $p$ we have a finite Fourier expansion
$$  \mathscr{I}_j(x)=\sum_{u \text{ mod } p} a_j(u)e_p(ux) $$
where $e_p(x)=e^{2\pi i x/p}$, and for $u=0,\ldots ,p-1$,
$$ a_j(u) = \frac{1}{p} \sum_{y \text{ mod } p} \mathscr{I}_j(y) e_p(-yu) = \begin{cases} p^{-1}N_j, 
 & \hbox{ if $u= 0$;} \\
p^{-1}  e_p(\xi_j u) \frac{\sin (\pi nu N_j/p )}{\sin(\pi nu/p)}, & \hbox{ if $u\neq 0,$  }\end{cases}$$
for some  $\xi_j$ in $\mathbb Z_p$.
Hence, separating into zero and nonzero values of $u$ and $v,$ and observing that $Ax^k$ is a permutation of $\mathbb Z_p$, we have
$$ N_{ij}  = \sum_{x=0}^{p-1} \sum_{ u=0}^{p-1} \sum_{v=0}^{p-1} a_i(u)e_p(ux) a_j(v) e_p(vAx^k)
   =M + T_1 + T_2+ E, $$
where
$$ M= p\ a_i(0)a_j(0)= p^{-1} N_iN_j,$$
\begin{align*} T_1 & =a_j(0)\sum_{u=1}^{p-1} a_{i}(u) \sum_{x=0}^{p-1}   e_p(ux) =0, \\
  T_2 & =a_i(0) \sum_{v=1}^{p-1} a_j(v) \sum_{x=0}^{p-1} e_p(vAx^k)= a_i(0) \sum_{v=1}^{p-1} a_j(v) \sum_{x=0}^{p-1} e_p(vx)=0, \end{align*}
and
$$ E =\sum_{u=1}^{p-1} \sum_{v=1}^{p-1} a_i(u)a_j(v) \sum_{x=0}^{p-1} e_{p}(ux+vAx^k). $$
Now from \cite[Theorem 1.3 ]{CochPin} we have, with $d=(k-1,p-1)$,
\be \label{binbound} \left| \sum_{x=0}^{p-1} e_{p}(ux+vAx^k)\right| \leq 1 + d + 2.292\ p^{89/92},  \ee
and from  \cite[Theorem 1]{Coch1},
observing that $nx$ is a permutation of the $x$ mod $p$,
\begin{align*}  \sum_{u=1}^{p-1} |a_j(u)|  & \leq \frac{1}{p} \sum_{x=1}^{p-1} \frac{|\sin(\pi x N_j/p )|}{|\sin( \pi x/p ) |} \\
 & \leq \frac{4}{\pi^2} \log p +.38+\frac{0.608}{p} +\frac{0.116}{p^3} \\
 &\leq \begin{cases} \frac{4}{\pi^2} \log p +.381, & \text{if $p>607$};\\
  \frac 4{\pi^2} \log p + \tfrac 12, & \text{if $p>5$}.
 \end{cases}
   \end{align*}
Hence for $p>607$,
\begin{align*}  |E|  & \leq (d+1+2.292\ p^{89/92}) \left(\sum_{u=1}^{p-1}|a_i(u)| \right)\left(\sum_{v=1}^{p-1}|a_j(v)| \right) \\
 &\leq  (d+1+2.292\ p^{89/92}) \left(\frac{4}{\pi^2} \log p +.381\right)^2. \end{align*}

\comments{Writing $p\equiv w$ mod $n$, $1\leq w <n$, we have for $(j,s)\neq (0,0)$
$$ M \geq \frac{1}{p} \left\lfloor \frac{p}{n}\right\rfloor^2 = \frac{1}{p} \left(\frac{p-w}{n}\right)^2\geq \frac{p}{n^2} -\frac{2w}{n^2}\geq \frac{p}{n^2}-\frac{1}{2}$$
and for $j=s=0$
$$ M-1= \frac{1}{p}\left( \left\lfloor \frac{p-1}{n} \right\rfloor +1\right)^2 -1 =\frac{1}{p} \left( \frac{p+n-w}{n}\right)^2-1 > \frac{p}{n^2}-1.$$
}

Since $p/n - 1 < N_j < p/n+1$ we have $p^2/n-1 < M < p^2/n+1$.
\end{proof}

Notice that if  $d<0.006 p^{89/92}$ and $p\geq e^{333}  (n\log n)^{184/3}$ then the main term in Theorem \ref{smallgcd} exceeds the error term and we can say that $f(I_i)\cap I_j\neq \emptyset$ for all $i,j$.
If our interest  is just in proving that $f(I_i) \cap I_j$ is nonempty, rather than obtaining an asymptotic estimate of its cardinality, then as shown in the next theorem we do not need  the  $\log n$ term.

\begin{theorem} \label{smallgcd2} Let $p$ be an odd prime and $A,k,n$ integers satisfying \eqref{range} with $2 \le n<p$,
and
$$ d=\gcd(k-1,p-1) \leq  0.006 p^{89/92}. $$

\noindent
(a) For any $i,j$, $0 \le i,j < n$, we have $f(I_i) \cap I_j \neq \emptyset$ provided that
$$
p>4 \cdot 10^{29}\; n^{\frac {184}3}.
$$

\noindent
(b) For any $i,j$, $0 \le i,j  < n$, we have  $f(I_i) \not\subseteq  I_j$ provided that
$$
p>9 \cdot 10^{34} \; n^{\frac {92}3}.
$$

\end{theorem}

\begin{proof} (a)  Recall, for $0 \le j \le n-1$, $I_j=\{x: \text {$x \equiv j$ mod $n$, $x\neq 0$}\} \subseteq \mathbb Z_p$, $N_j = |I_j|$.
Let
\begin{align*}
J_j & := \{j, j+n, \dots, j+ (\lceil N_j/2\rceil -1) n\} \subseteq \mathbb Z_p, \quad j\neq 0,\\
K_j & := \{0,n,2n, \dots, \lfloor N_j/2\rfloor n\}\subseteq \mathbb Z_p,
\end{align*}
with $J_0=\{n,2n,\dots, \lceil N_0/2\rceil  n\}$, so that $J_j+K_j \subseteq I_j$, and let
$\alpha_j = \mathscr I_{J_j}*\mathscr I_{K_j}$, the convolution of the characteristic functions of $J_j$ and $K_j$,
$$
\alpha_j(x):= \underset{u+v=x}{\sum_{u \in J_j}\sum_{v \in K_j }}1,
$$
with Fourier coefficients $b_j(y)$ say. Then $\alpha_j$ is supported on $I_j$, and so our goal is to show that for any $i,j$,
$$
\sum_{x \text{ mod } p} \alpha_i(x) \alpha_j(Ax^k) >0.
$$
Expanding the sum as before we obtain
\begin{align*}
&\sum_{x \text{ mod } p} \alpha_i(x) \alpha_j(Ax^k)\\
&= b_i(0)b_j(0)p + \sum_{u\neq 0}\sum_{v \neq 0} b_i(u)b_j(v) \sum_{x \text{ mod } p} e_p(ux+vAx^k)\\
&= M' + E',
\end{align*}
say, where $M'$ is the main term and $E'$ the error term. Plainly, we have
$$
M'=|J_i||K_i||J_j||K_j|p^{-1}.
$$
Next, using the fact that $b_j(v) = p a_j'(v)a_j''(v)$ where $a_j'(v)$, $a_j''(v)$ are the Fourier coefficients of $\mathcal I_{J_j}$, $\mathcal I_{K_j}$, we obtain from the Cauchy-Schwarz inequality and Parseval identity that for any $j$,
\begin{align*}
\sum_{v \text{ mod }p} |b_j(v)| &= p \sum_{v \text{ mod } p} |a_j'(v)||a_j''(v)| \le p \left(\sum_{v \text{ mod } p} |a_j'(v)|^2\right)^{\frac 12} \left(\sum_{v \text{ mod } p} |a_j''(v)|^2\right)^{\frac 12} \\
&= |J_j|^{\frac 12}|K_j|^{\frac 12}.
\end{align*}
Thus, since $d\leq 0.006p^{89/92}$ and $p>10^{29}$,
$$
|E'|\le (d+1+ 2.292\ p^{89/92}) |J_j|^{\frac 12}|K_j|^{\frac 12}|J_i|^{\frac 12}|K_i|^{\frac 12} < 2.299\ p^{89/92} |J_j|^{\frac 12}|K_j|^{\frac 12}|J_i|^{\frac 12}|K_i|^{\frac 12},
$$
and we see that $M'>|E'|$ provided that
\be \label{thm3.1typeiv}
|J_j||K_j||J_i||K_i| > p^2\left(2.299\ p^{89/92}\right)^2.
\ee
Since $|J_j| \ge N_j/2$, $|K_j| \ge N_j/2$ and $N_j \ge \frac pn -1$ for all $j$, it suffices to have
$$
(p-n)^4 > 2^4n^4p^2 \left(2.299\ p^{89/92}\right)^2,
$$
and for this it suffices to have
$
p> 4\cdot 10^{29} n^{\frac {184}3}.
$

\vspace{2ex}
(b) We may assume that $n\geq3$. For type (iii) intersections, we let $I_j^c= \mathbb Z_p \setminus I_j$, a set of cardinality $N_j^c =p-|I_j| \ge p(1-\frac 1n)-1\geq 2p/3 -1$. For $j\neq 0$ we shall think of $I_j^c$ as the arithmetic progression
$$ I_j^c=\{j+nt \; : \; t=N_j,N_j+1,\ldots ,p-1\}, $$
on observing that the values corresponding to $t=0,1,\ldots ,p-1$ are distinct mod $p$, giving $\mathbb Z_p,$ and that we have removed the $t=0,1,\ldots , N_j-1$ constituting $I_j$.
Similarly  $I_0^c=\{nt \; : \; t=N_0+1,\ldots ,p\}$. We define $\alpha_j$ as above with sets
\begin{align*}  J_j' & =\{  j+N_j n,j+(N_j+1)n,\ldots , j+(N_j+\lceil N_j^c/2\rceil -1)n \},\; j\neq 0,\\
K_j' & =\{0,n,2n,\ldots ,\lfloor N_j^c/2\rfloor n\}, \end{align*}
with $J_0'=\{ (N_j+1)n,\ldots (N_j+\lceil N_j^c/2\rceil )n\}$, so that $J_j',K_j'$ have cardinalities at least $N_j^c/2$, and  $\alpha_j$ is supported on $I_j^c$.  Once again we succeed in obtaining $f(I_i)\cap I_j \neq \emptyset$, provided that
\be \label{thm3.1typeiii}
|J_j'||K_j'||J_i||K_i| >p^2 (2.299 p^{89/92})^2,
\ee
and for this it suffices to have
\begin{equation} \label {typeiiisize}
p>9\cdot 10^{34} \; n^{\frac {92}3}.\qedhere
\end{equation}

\end{proof}

\section{Type $(iii)$ intersections for large $d$ }\label{dlarge(iii)}

In this section we show that for large $d$ we cannot have $f(I_i)\subseteq I_j$ for any $i,j$ provided $p$ is sufficiently large. Recall $$I=\{1,2,\dots, p-1\}, \quad \quad I_j:\{x\in I: x \equiv j \text{ mod } n\}.$$

\begin{theorem}\label{biggcd} Suppose that $f(x)\neq \pm x$  mod $p$ when  $n$  is even, and $f(x)\neq \pm x$ or $\pm x^{\frac{1}{2}(p+1)}$ mod $p$ when $n$ is odd. If
 $p>10^6 $
and $d\geq 0.66 n p^{1/2}\log^2 p$,  then $f(I_i)\cap (I\setminus I_j) \neq \emptyset$ for all $i,j$.

The same conclusion holds if  $p\ge 7$ and $d>3n\sqrt{p}\log^2 p$.
\end{theorem}


Plainly  $f(x)=\pm x$ mod $p$ maps $I_i$ to $I_i$ or to  $I_{\overline{i}}$ where $\overline{i}\equiv p-i$ mod $n$ so must be excluded. The $f(x)=\pm x^{(p+1)/2}$ are dealt with in Example \ref{ex1}, where for $p>4n^2$ and $n$ even we always have $f(I_i) \cap (I\setminus I_j) \neq \emptyset$, but for odd $n$ must be excluded.

 \begin{proof} Suppose that $(A,k)\neq (\pm 1,1)$ or $(\pm 1,(p+1)/2)$.
Observe that  the set of absolute least residues
$$  \mathscr{C}= \{C=A x^{k-1} \text{ mod } p \; :\; 1\leq x\leq p-1,\;\;\ |C|<p/2\},   $$
must contain at least one element $C\neq \pm 1$. To see this observe that $\mathscr{C}$ contains  $(p-1)/d$ elements
and hence more than two unless $d=(p-1)$ or $(p-1)/2$ and $k=1$ or $(p+1)/2.$  In these cases  $\mathscr{C}$  contains only $A$ or $\pm A$  and we just need to avoid $A=\pm 1$. We need to prove that $f(I_i)\cap (I\setminus I_j)\neq \emptyset$.
 We shall suppose that our $C\equiv AB^{k-1}$ mod $p$  satisfies $1<C<p/2$; if all the potential $C$'s are negative we replace $A$ by $-A$ and $j$ by the least residue of $p-j$ mod $n$.
 We let
$$ L:=(p-1)/d $$
and
$$\mathscr{U}=\{ x\in I_i \: : \;\; Cx \text{ mod } p \in I\setminus I_j, \;\; x\equiv Bz^L \text{ mod } p \text{ for some $z$}\}. $$
Notice that if $x$ is in $\mathscr{U}$ we have
$$ Ax^k\equiv Cx (B^{-1}x)^{k-1} \equiv Cx z^{L(k-1)}=Cx (z^{p-1})^{(k-1)/d} \equiv Cx \text{ mod } p $$
and we have an $f(x)$ in $f(I_i)\cap  (I\setminus I_j)$. So it is enough to show that $|\mathscr{U}|>0$.
Let $\hat{G}$ denote the set of Dirichlet   (multiplicative) characters on $\mathbb Z_p^*$ with principal character $\chi_0$ and recall that
$$ \sum_{\chi \in \hat{G}, \chi^{L}=\chi_0} \chi (y) = \begin{cases} L, & \hbox{ if $y$ is an $L$th power mod $p$,} \\
0, & \hbox{ if $y$ is not an $L$th power mod $p$. } \end{cases} $$

Hence, writing $\mathscr{I}^c_j(x) $ for the characteristic function of $I\setminus I_j$, the complement of $I_j$, we have
$$ L|\mathscr{U}|= \sum_{x\in \mathbb Z_p^*}  \mathscr{I}_i(x)\mathscr{I}^c_j(Cx)     \sum_{\chi \in \hat{G}, \chi^{L}=\chi_0} \chi (B^{-1}x). $$
Separating the principal character from the remaining $L-1$ characters with $\chi^L=\chi_0$
$$ L|\mathscr{U}|= M+E, $$
where $M$ is our `main term'
$$ M= \sum_{x\in \mathbb Z_p^*}  \mathscr{I}_i(x)\mathscr{I}^c_j(Cx), $$
and $E$ the `error'
$$ E= \sum_{\chi^L=\chi_0, \chi\neq \chi_0} \chi(B^{-1}) S(\chi), $$
with $E=0$ when $k=1$, where
$$ S(\chi)= \sum_{x\in \mathbb Z_p}  \chi(x)\mathscr{I}_i(x)\mathscr{I}^c_j(Cx).$$

\noindent
{\bf Error Term.}  Taking the finite  Fourier expansion for the intervals  as in the proof of Theorem \ref{smallgcd} we have as before
$$ \mathscr{I}_i(x)= \sum_{ y\in \mathbb Z_p} a_i(y) e_p(yx), \;\;\; |a_i(y)| =\frac{1}{p} \begin{cases}  N_i, & \hbox{ if $y=0$; }   \\  \frac{|\sin (\pi N_i ny/p)|}{|\sin(\pi ny/p)|}, &    \hbox{ if $y\neq 0$, } \end{cases}$$
with $N_i =|I_i|$,  and
$$ \mathscr{I}^c_j(x)= \sum_{ y \in \mathbb Z_p} a_j^c(y) e_p(yx), \;\;\; a_j^c(y) = \begin{cases} 1-a_j(0), & \hbox{ if $y=0$; }   \\  -a_j(y), &    \hbox{ if $y\neq 0$. } \end{cases}$$
Again, separating the terms with $u$ or $v$ zero, we have
$$ S(\chi) =  \sum_{x\in \mathbb Z_p}  \chi(x) \sum_{u=0}^{p-1} a_i(u)e_p(ux) \sum_{v=0}^{p-1} a_j^c(v)e_p(vCx)= T_1+E_1+E_2+E_3 $$
where
\begin{align*}
  T_1  & = a_i(0)a_j^c(0) \sum_{x\in \mathbb Z_p} \chi(x)=0, \\
  E_1 & =a_i(0) \sum_{v=1}^{p-1} a_j^c(v) \sum_{x=0}^{p-1} \chi (x) e_p(Cvx), \\
  E_2 & =a_j^c(0) \sum_{u=1}^{p-1} a_i(u) \sum_{x=0}^{p-1} \chi (x) e_p(ux),
\end{align*}
and
$$ E_3 = \sum_{u=1}^{p-1}\sum_{v=1}^{p-1} a_i(u)a_j^c(v) \sum_{x\in \mathbb Z_p} \chi(x)e_p((u+Cv)x). $$
Recalling that,  for a non-principal character $\chi$, the classic Gauss sums
$$ G(\chi,A)= \sum_{x=0}^p \chi (x) e_p(Ax) $$
satisfy $|G(\chi,A)|=p^{1/2}$ if $p\nmid A$ and trivially $G(\chi,A)=0$ if $p\mid A$, and again invoking \cite[Theorem 1]{Coch1}, we have for $p>607$
\begin{align*} |E_1| &\leq \frac{N_i}{p}  \sum_{v=1}^{p-1}|a_{j}^c(v)| p^{1/2} \leq  \frac{N_i}{p} \left(\frac{4}{\pi^2}\log p+0.381\right) p^{1/2}, \\
 |E_2| &\leq \frac {(p-1-N_j)}{p}  \sum_{u=1}^{p-1}|a_{i}(u)| p^{1/2} \leq   \frac{(p-1-N_j)}{p}  \left(\frac{4}{\pi^2}\log p+0.381\right) p^{1/2}, \\
 |E_3| &\leq \left( \sum_{u=1}^{p-1} |a_i(u)|\right)\left( \sum_{v=1}^{p-1} |a_j^c(v)|\right)  p^{1/2} \leq \left(\frac{4}{\pi^2} \log p +0.381\right)^2  p^{1/2},
 \end{align*}
with $N_i+(p-1-N_j) \leq p$. Hence, for $p>10^6$,
\begin{align*}  |S(\chi)|  &  \leq  \left(\frac{4}{\pi^2}\log p+0.381\right) p^{1/2} + \left(\frac{4}{\pi^2} \log p +0.381\right)^2  p^{1/2} < 0.22 \; p^{1/2}\log^2 p, \end{align*}
and
$$  |E|<  0.22 ( L-1)  p^{1/2}\log^2 p.  $$

\vspace{1ex}
\noindent
{\bf Main Term.} We have
$$ M= |I_i| - \sum_{x\in \mathbb Z_p^*}  \mathscr{I}_i(x)\mathscr{I}_j(Cx)
= N_i - M_{ij}, $$
where $N_i$ is as given in \eqref{Nj}, and
$$ M_{ij}=|\{ x\in I_i\; : \; Cx \text{ mod } p \in I_j\}|. $$
So for a lower bound on $M$ we need an upper bound on $M_{ij}$.
Since for $1\leq x<p$ we have $0<Cx <Cp$ we have
$$ M_{ij} = \sum_{u=0}^{C-1} |\{ x\in I_i \; : \; up\leq Cx < (u+1)p,\;\; Cx-up \in I_j\}|. $$
Note, if $x\equiv i$ mod $n$ then  $Cx-up\equiv j$ mod $n$ requires $u\equiv K:=(Ci-j)p^{-1}$ mod $n$. Observing that the number of elements  in a particular residue class mod $n$ in an interval of cardinality $B$ is at most $\lfloor (B-1)/n \rfloor +1$ we have
\begin{align}  M_{ij} & = \sum_{\substack{u=0\\u\equiv K \text{ mod } n}}^{C-1} \left|\left\{   x\in I_i \;  : \; \frac{up}{C} \leq x < \frac{up}{C} +\frac{p}{C}     \right\}\right| \notag \\
 & \leq  \left( \left\lfloor \frac{C-1}{n}\right\rfloor  + 1\right)  \left(  \left\lfloor \frac{p/C  }{ n } \right\rfloor +1          \right).  \label{mij1} \end{align}
Plainly
\begin{equation} \label{mij2}  M_{ij} \leq \left(\frac{C}{n}+1\right)\left(\frac{p}{Cn} +1\right) =\frac{p}{n^2} +\frac{C}{n} + \frac{p}{Cn} +1, \end{equation}
and so for $p/2n\geq C \geq 2n$,
$$M_{ij} \leq  \frac{2p}{n^2} + 1. $$
For $2n> C\geq n$, $\lfloor (C-1)/n\rfloor=1$ and so by \eqref{mij1},
$$
M_{ij} \le 2\left( \frac {p}{Cn} +1 \right) \le 2\left( \frac p{n^2} +1\right) =\frac {2p}{n^2} +2,
$$
while for $p/n\geq C> p/2n$, $\lfloor p/(Cn)\rfloor =1$, and so
$$
 \le 2\left(\frac Cn +1\right) \le 2\left(\frac p{n^2} +1\right) = \frac {2p}{n^2} +2.
$$
For $C<n$, since $2\leq C < p/2$, we have by \eqref{mij1},
$$   M_{ij}\le  \left( \left\lfloor \frac{C}{n}\right\rfloor  + 1\right)  \left(  \left\lfloor \frac{p/C  }{ n } \right\rfloor +1          \right) \leq 1\cdot \left(\frac{p}{Cn}+1\right) \leq \frac{p}{2n}+1, $$
and when $C>p/n$
$$ M_{ij} \le \left( \left\lfloor \frac{C}{n}\right\rfloor  + 1\right)  \left(  \left\lfloor \frac{p/C  }{ n } \right\rfloor +1          \right) \leq \left(\frac{C}{n}+1\right) \cdot 1 < \frac{p}{2n}+1. $$
Hence, in all cases we have
\begin{equation} \label{Mijubb}
M_{ij} \le \max \left\{ \frac {2p}{n^2} +2,\quad  \frac p{2n}+1\right\},
\end{equation}
and see that for $n\geq 3$,
$$ M_{ij}\leq \frac{2p}{3n}+2, $$
and
\be \label{bigp}  M\geq \left\lfloor \frac{p}{n} \right\rfloor - M_{ij} > \frac{p}{n}-1 -M_{ij}\geq \frac{p}{3n}-3. \ee
Thus if  $p/3n \geq  (0.22 p^{3/2}\log^2 p)/d$  we have $|E|<M$ and $|\mathscr{U}|>0$.

If instead, we use the bound $\sum_{u} |a_i(u)| \le \frac 4{\pi^2} \log p + \frac 12$, valid for $p \ge 7$, we obtain
$|E| < L \sqrt{p} \log^2 p -4$, and conclude that $M>|E|$ for $d \ge 3n \sqrt{p} \log^2 p$.
\end{proof}

\section{Proofs of  Theorems \ref{main}, \ref{mainiiik=1}, \ref{lehmer}} \label{proofmain}

\begin{proof}[Proof of Theorems \ref{main}]   Suppose that $p> 9\cdot 10^{34} \:n^{92/3}$. Then certainly $p>6.7\times 10^8$. If $d\leq 0.006\: p^{89/92}$ then Theorem \ref{main} follows from Theorem  \ref{smallgcd2}, while if $d\geq   0.66 \: n p^{1/2}\log^2 p$
it follows from Theorem \ref{biggcd}.  If neither of these occurs then $ 0.66 \:n p^{1/2}\log^2 p>d> 0.006 \:p^{89/92}$ and so
$p^{43/92}/\log^2 p <110 n.$ But this  does not occur for $p>9 \cdot 10^{34}\: n^{92/3}$.
\end{proof}

\begin{proof}[Proof of Theorem \ref{mainiiik=1}]
We revisit the proof of  Theorem   \ref{biggcd}. For $k=1$  there is no error term $E$, and so we need only show that $M>0$. This follows from \eqref{bigp} for $p>9n$. Our computations, see Table \ref{tablei}, have checked $2n<p<9n$ for $3\leq n\leq 12$ so we can assume that $n>12$. For $4n<p<9n$ we plainly have $2p/n^2 +2<18/n+2<4\leq \lfloor p/n\rfloor$ and
$p/2n+1 <p/n-1$ and thus by \eqref{Mijubb} $N_i>M_{ij}$.  Finally for $2n<p<4n$  by \eqref{Mijubb} we have $M_{ij}<3$ and, since $M_{ij}$ is a count,  $M_{ij}\leq 2$.
Hence the result when $p>3n$ and $\lfloor p/n\rfloor \geq 3$, or when $2n<p<3n$ and $|I_i|=3$.

 It remains to check the case  $2n<p<3n$ when  $|I_i|=2$. Writing $p=2n+e$ with $1\leq e<n$,  and $I_n$ for  $I_0,$ we have $I_j=\{j,j+n,j+2n\}$ for $1\leq j\leq e$, and $I_j=\{j,j+n\}$ for $e<j\leq n$. Suppose $f(I_i)\subseteq I_j$ with $e<i\leq n$, $1\leq j\leq n$. We assume that $A\neq \pm 1$ and that $A>0$ (else replace $A$ by $-A$ and $j$ by $\bar{j}=p-j$ mod $n$).
Then
$$
f(i)\equiv Ai \equiv j+un \text{ mod }p,\;\;\;\;\;
f(i+n)\equiv Ai+An \equiv j+vn \text{ mod }p, $$
for some $u \neq v \in \{0,1,2\}$.  Subtracting, we get $A\equiv v-u$ mod $p$, and since $A\geq 2$, get $A=2$, $v=2$, $u=0$. This yields $2i \equiv j$ mod $p$, meaning $j=2i>i$ since $p>2n$.  But, $v=2$ implies that $|I_j|=3$ and hence $j<i$.
\end{proof}

\comments{
\begin{proof}[Proof of Theorem \ref{k=(p+1)/2}] Suppose that $k=(p+1)/2$ and $A\neq \pm 1$.  In the proof of Theorem \ref{biggcd} we have $L=(p-1)/d=2$ and $|E|<0.22 p^{1/2}\log^2p$. Hence we just need $p>10^6$
and $p/(3n)-3>0.22 p^{1/2}\log^2p,$ which holds for $p>4275 n^2 \log^4 n$.
\end{proof}
}

\begin{proof}[Proof of Theorem \ref{lehmer}] In the proof of Theorem \ref{smallgcd2} we use the Weil bound \cite{Weil}
$|k-1|\sqrt{p}$  in place of  \eqref{binbound} and for (a) and (b) we  just need
 $$
|J_j'||K_j'||J_i||K_i| >|k-1|^2p^3  \;\;\text{ and } \;\; |J_j||K_j||J_i||K_i| >|k-1|^2p^3,
$$
in place of \eqref{thm3.1typeiii} and \eqref{thm3.1typeiv}.
\end{proof}

\section{Proof of Theorem \ref{square}}\label{Squaresection}

Notice that $f(x)=Ax^{(p+1)/2}$ mod $p$  is related to two linear maps:
\be \label{reduction}  f(x)\equiv A\left(\frac{x}{p}\right) x \equiv \pm Ax \quad \text{ mod } p, \ee
and that the inverse mapping $f^{-1}(x)$ is given by
\begin{equation} \label{finv} f^{-1}(x)=\left(\frac{A}{p}\right) A^{-1} x^{(p+1)/2}\quad \text{ mod } p.
\end{equation}
In order to prove $f(x)$ is not a Type (iii) mapping we can replace $f(x)$ with $\pm f(x)$ or $\pm f^{-1}(x)$ (with one exception), which amounts to changing $A$ to $\pm A$ or $\pm A^{-1}$. Thus we define the quantity
\begin{equation} \label{Cdef}
C:= \min \{|A|, | A^{-1}|\},
\end{equation}
where  $A, A^{-1}$ are taken to be integers with $|A|,|A^{-1}|<p/2$.  Note that if $|I_i|=|I_j|$ then $f(I_i)\subseteq I_j$ is the same as $f(I_i)=I_j$ or $f^{-1}(I_j)=I_i$.
     The one exception that needs special attention is if for some $i,j$,  $f(I_i)$ is a proper subset $I_j$, that is,
 $|I_j|=|I_i|+1$. Then $f(I_i)=I_j\setminus \{a\}$  for some $a$, with $f^{-1}(I_j\setminus \{a\})=I_i$, and so in replacing $f$ with $f^{-1}$ we must remove one element from each of the larger $I_i$ and still show that $f(I_i)$ hits at least two different residue classes.  
 
 To prove Theorem \ref{square} we must show that for $p>(4n+1)^2$, and $C \ge 2$, the mapping \eqref{reduction} is not a Type (iii) mapping.
The theorem is an immediate consequence of the following two lemmas, the first dealing with the case $C=2$, and the second all larger $C$.

\begin{lemma} \label{awkward2} Suppose that $n \ge 2$,  $f(x)=Ax^{(p+1)/2}$ mod $p$ with $A \neq \pm 1$, and let $C$ be as given in \eqref{Cdef}.
If $p>(2Cn+1)^2$ then $f(x)$ is not a Type (iii) mapping.
\end{lemma}

\begin{proof}  Suppose first that $A=C$ and $0\leq i<n$. Consider the sets
\begin{align*} U_1 & =\{ u\in \mathbb Z\; :\; 0\leq  u < (p/C-i)/n\}, \\
U_2 & =\{ u\in \mathbb Z\; :\; (p/C-i)/n \leq u \leq (2p/C-i)/n\}, \end{align*}
with $u=0$ excluded from $U_1$ when $i=0$.
Since $p>(2Cn+1)^2$ we have
$$ |U_i|\geq \frac{p}{Cn}-1 >2\sqrt{p}, $$
and thus by the result of Hummel \cite{Hummel}, any translate of these intervals must contain at least two quadratic residues and two nonresidues.
 Hence there will be  $u_1,u_2$ in $U_1$ and $u_3,u_4$ in $U_2$ with 
  $$
\left( \frac{in^{-1}+u_l}{p}\right)=\left(\frac{n}{p}\right), \quad \quad l=1,2,3,4,
$$
and therefore
$$
\left(\frac{i+u_ln}{p}\right)=1, \quad \quad l=1,2,3,4.
$$
 Note that
 \begin{align*}
0&<i+u_1 n<p/A, \quad \quad l=1,2,\\
p/A &<i+u_ln<2p/A\leq p, \quad \quad l=3,4,
\end{align*}
and thus $i+u_ln \in I_i$, $1 \le l\le 4$ with by \eqref{reduction},
\begin{align*}
 f(i+u_l n) & =C(i+u_ln)\equiv Ci \text{ mod } n,\quad \quad l=1,2\\
  f(i+u_ln)& = C(i+u_ln)-p  \equiv Ci-p \text{ mod } n, \quad \quad l=3,4.
  \end{align*}
These two values must be distinct mod $n$. 

Finally, if $A \neq C$ then we replace $f(x)$ with $-f(x), f^{-1}(x)$ or $-f^{-1}(x)$ to make $A=C$, and note that passing to $f^{-1}(x)$ presents no new difficulties because each $I_i$ had at least two quadratic residues and two quadratic nonresidues.
 \end{proof}

\begin{lemma} 
\label{squarelemma}  Suppose that
$ f(x)=Ax^{(p+1)/2} \text{ mod } p,$ with $A \neq \pm 1$, 
$n\geq 2$ and $p>9n^2$. Then $f(x)$ is not a Type (iii) map.
\end{lemma}

\begin{proof}  Suppose first that $n \ge 4$ and that $C$ is as given in \eqref{Cdef}.  Lemma \ref{awkward2} dispenses with the case $C=2$ and so we assume $3 \le C <p/2$, $p>9n^2$.
\comments{For $n=3$, $k=(p+1)/2$ we have $L=2,$ and  the proof of Theorem 4.1 gives $M>p/9-3$, where
$$ |E|\leq \left( \frac{4}{\pi^2} \log p+0.381\right)p^{\frac{1}{2}}+ \left( \frac{4}{\pi^2} \log p+0.381\right)^2p^{\frac{1}{2}}< \frac{p}{9}-3, $$
for $p\geq 66975$.
}
Writing $\overline{j}=p-j$ mod $n$, and
$$ M_{ij}=\abs{\{x\in I_i \; : \; Ax \text{ mod } p \in I_j\}}, $$
if $f(I_i)\subseteq I_j$ then by \eqref{reduction} the image of $I_i$ under the linear map $Ax$ mod $p$ must lie in
$I_j$ or $I_{\overline{j}}$ and we must  have $|I_i|\leq M_{ij} +M_{i\overline{j}}.$
Hence to rule out a Type (iii) it will be enough to check that for all $i,j$
\be \label{2M}  2M_{ij} < \frac{p}{n}-1. \ee
We proceed as in the proof of Theorem \ref{biggcd} considering various ranges for the size of $C$.
In the cases where we need to replace $f$ with $f^{-1}$ and delete one element from the larger $I_i$, we  have for some $j$, $|I_i|-1=|I_j|>p/n-1$  and so it is still
enough to show \eqref{2M} for $f^{-1}$.

\vspace{1ex}
\noindent
{\bf  High $C$'s.} For $n\geq 5$ we claim that we cannot have $C$ in the range
 $$
  \displaystyle \frac{p}{2}-\frac{1}{2}\sqrt{p-4} <  C \leq \frac{p}{2}.
  $$
Indeed, in this case either $C^{-1}$ or $-C^{-1}$ is also in this range, say $C':=\pm C^{-1}$. Then
$$ \pm 4 \equiv  (p-2C)(p-2C') \quad \text{ mod } p, $$
but $| (p-2C)(p-2C') \pm 4|<4+\sqrt{p-4}^2=p$ and parity rules out equality.
Similarly, for $n=4$ we claim that we cannot have $C$ in either of the intervals   
 $$ \frac{p}{s} -\frac{1}{4}\sqrt{p-16} < C \leq \frac{p}{s}, \quad \quad s=2,4.
$$
Indeed, in this case with $C'$ as defined before for some $s' \in \{2,4\}$,
\begin{align}
\pm 16 &\equiv \left(\frac 4s p-4C\right)\left(\frac 4{s'} p -4C'\right)\quad  \text{mod $p$}, \label{what}\\
0 &<\left(\frac 4s p-4C\right)\left(\frac 4{s'} p -4C'\right)< p-16,\notag
\end{align}
implying equality in \eqref{what} which cannot occur since $4 \nmid \frac 4s p-4C$ or  $\frac 4{s'}p-4C$.

\vspace{1ex}
\noindent
For the remaining ranges we use the inequality in \eqref{mij1}, \eqref{mij2},
\begin{equation} \label{mij3}
M_{ij} \le \left( \left\lfloor \frac{C-1}{n}\right\rfloor  + 1\right)  \left(  \left\lfloor \frac{p/C  }{ n } \right\rfloor +1          \right)
     \le \frac{p}{n^2} +\frac{C}{n} + \frac{p}{Cn} +1. 
\end{equation}

\noindent
{\bf  Middle  $C$ values.} Suppose that $2n \leq C \leq p/2n$.
Since $p> 9n^2$ we have  $2n\leq \sqrt{p} \leq p/2n$ and
$$ M_{ij} <\frac{p}{n^2}+\frac{C}{n}+ \frac{p}{Cn}+1 \leq \frac{3p}{2n^2}+3\leq \frac{3p}{8n}+3. $$
Hence,
$$ 2M_{ij} <\frac{3p}{4n}+6 \leq \frac{p}{n}-1, $$
for $p\geq 28n$, which holds for $p>9n^2$ and $n \ge 4$.  

\vspace{1ex}
\noindent
{\bf  Small  $C$ values.} Suppose that $n < C < 2n$. Then by \eqref{mij3}, 
$$ M_{ij}  < 2\left( \frac{p}{Cn}+1\right) \leq \frac{2p}{5n} +2 $$
and $2M_{ij}<p/n-1$ for $p>25n$, which holds as before.


\vspace{1ex}
\noindent
{\bf  Very Small   $C$ values.} Suppose that $3\leq C\leq n$. Then by \eqref{mij3}, 
$$ M_{ij}  < \left( \frac{p}{Cn}+1\right) \leq \frac{p}{3n} +1 $$
and $2M_{ij} < 2p/3n+2 <p/n-1$  as long as $p>9n$.

\vspace{1ex}
\noindent
{\bf  Large $C$ values.} Suppose that $n\geq 5$ and  $p/n < C < p/2-\frac{1}{2}\sqrt{p-4}$,
or $n=4$ and $p/n <C < p/2-\frac{1}{4}\sqrt{p-16}$.
For $n\geq 5$ we get by \eqref{mij3},
$$ M_{ij}  \leq  \left( \frac{C-1}{n}+1\right) \cdot 1 \leq  \frac{p}{2n} - \frac{\sqrt{p}}{2n} +1$$
and $2M_{ij}<(p/n)-1$ for $p>9n^2$.

For $n=4$ we have $2M_{ij} \leq \frac{p}{n}- \frac{\sqrt{p}}{2n} +2 < \frac{p}{n}-1$ for $p>36n^2$, but there
are no values giving Type (iii) with $k=(p+1)/2$ and $p<576$.

\vspace{1ex}
\noindent
{\bf  Largish  $C$ values.} Suppose that  $n\geq 5$ and $p/2n< C \leq p/n$,
or $n=4$ and $p/2n< C\leq p/4-\frac{1}{4}\sqrt{p-16}$).
If  $n\geq 5$  then by \eqref{mij3},
$$ M_{ij}  \leq  \left( \frac{C-1}{n}+1\right) \cdot 2 \leq  2\frac{p}{n^2} +2\leq \frac{2p}{5n}+2$$
and $2M_{ij}<(p/n)-1$ for $p>25n$.

For $n=4$ we get from \eqref{mij3}, 
$$ 2M_{ij}\leq \frac{p}{4} -\frac{1}{4}\sqrt{p} +4 < \frac{p}{4}-1 $$
for $p>20^2$.  There are no examples with $p<400$.

\vspace{2ex}
\noindent
{\bf The Case $n=3$.} 
 It remains to deal with $n=3$. From our computations we know that there are no Type (iii) mappings with $6<p<20,000$. Replacing $A$ by $A^{-1}$ as necessary we may assume that our $C$
does not lie in any of the intervals
$$ U_s=\left(\frac{p}{s}-\frac{1}{12}\sqrt{p-144}\;, \; \frac{p}{s}\right), \;\;\;s=2,4 \hbox{ or } 6.$$
To see this observe that  if $C=\pm A$ is in $U_s$ and $C'=\pm A^{-1}$ is in $U_{s'}$ then
$$ \pm 144 \equiv (12p/s-12C)(12p/s'-12C') \text{ mod  }p,   \;\;\;0< (12p/s-12C)(12p/s'-12C')<p-144, $$
where $2^2\nmid 12p/s-12C$ or $12p/s'-12C'$ rules out equality.

For $2\leq C \leq 9$ from Lemma \ref{awkward2} there are no such Type (iii) with $p>3025$.

For $9\leq C\leq p/9$  and $p>243$ we have
$$2M_{ij} \leq 2\left(\frac{p}{9}+\frac{C}{3}+\frac{p}{3C}+1\right) \leq \frac{8p}{27}+8<\frac{p}{3}-1.$$

For $p/9 < C < p/6 - \frac{1}{12}\sqrt{p-144}$  and $p>1764 $ we have
$$ 2M_{ij}\leq 2\left(  \frac{C-1}{3}+1  \right) \cdot 3 \leq   \frac{p}{3}-\frac{\sqrt{p}}{6} +6 <\frac{p}{3}-1. $$

For $p/6 < C < p/4 - \frac{1}{12}\sqrt{p-144}$  and $p>2025 $ we have
$$ 2M_{ij}\leq 2\left(  \frac{C-1}{3}+1  \right) \cdot 2 \leq   \frac{p}{3}-\frac{\sqrt{p}}{9} +4 <\frac{p}{3}-1. $$

For $p/3 < C < p/2 - \frac{1}{12}\sqrt{p-144}$  and $p>2916 $ we have
$$ 2M_{ij}\leq 2\left(  \frac{C-1}{3}+1  \right) \cdot 1 \leq   \frac{p}{3}-\frac{\sqrt{p}}{18} +2 <\frac{p}{3}-1. $$

That just leaves the case where $p/4<C<p/3$. We deal with the map
$$ g(x)=Cx^{(p+1)/2}\equiv\left(\frac{x}{p}\right) Cx \text{ mod } p $$
directly on $I_0=\{3,6,9,\ldots\}$, $I_1=\{1,4,7,\ldots\}$ and $I_2=\{2,5,8,\ldots\}$.

For $I_0$ observe that $g(6)=6C-p$ or $2p-6C\equiv 2p$ mod $3$ while $g(9)=9C-2p\equiv p$ mod 3 giving us an element in $I_1$ and an element  in $I_2$.

For $I_1$ we have $g(1)=C$ and $g(4)=4C-p$ and these are distinct mod $3$.

For $I_2$ we have
\begin{align*} \left(\frac{2}{p}\right) & =1 \;\; \Rightarrow \;\;g(2)=2C,\;\; g(8)=8C-2p,\\
 \left(\frac{2}{p}\right) & =-1 \;\; \Rightarrow \;\;g(2)=p-2C,\;\; g(8)=2p-8C, \end{align*}
and in either case these are distinct mod $3$.

This deals with the case  $C=\pm A$ and $f(x)=\pm g(x)$ or if $C=A^{-1}$ in the case when $f(I_i)\subseteq I_j$
and the $|I_i|=|I_j|$ as must happen when $p\equiv 1$ mod 3. This just  leaves the case where $p\equiv 2$ mod 3 and $f(I_0)$ or $f(I_2)$ equals $I_1\setminus\{a\}$ when the missing $a=1$ or $4$.
 But  notice that when $p\equiv 2$ mod $3$ we have $(p-1)$ and $(p-4)$ in $I_1$ where
$g(p-x)\equiv -g(x) \text{ mod }p$. Hence $g(p-1) =p-C$ and $g(p-4)=2p-4C$ and these two values are again distinct mod 3.
\end{proof}

\section{Proof of Example \ref{ex1}}\label{ProofEx}

\begin{proof}[Proof of Example \ref{ex1}] Suppose that $f(x)=\pm x^{(p+1)/2}$ mod $p$. We have
$$x^{(p+1)/2} =x\cdot  x^{(p-1)/2}\equiv x \left(\frac{x}{p}\right) \equiv \pm x \text{ mod } p, $$
and  $f(x)=x$ or $p-x$, where $(p-x)\equiv x$ mod $n$ exactly when  $x\equiv2^{-1}p$ mod $n$
if  $n$ is odd and in no cases if  $n$ is even, and the first claim  is plain.

If $n$ is even, or $n$ is odd and $i \neq 2^{-1}p$ mod $n$, then $x\not\equiv  p-x$ mod $n$ for $x$ in $I_i$, and $f(I_i)$ will hit  two different residue classes as long as  $I_i$ contains both quadratic
residues and nonresidues.  Suppose that $\left(\frac{x}{p}\right)$ is constant on $I_i$, then $\left(\frac{n^{-1}i+y}{p}\right)$
is constant for $y$ in an interval of length $|I_i|$. But by \cite{Hummel} there are less than $\sqrt{p}$ consecutive residues or nonresidues, and for $p>(n+1)^2$ we have $|I_i|>p/n -1 >\sqrt{p}$. \end{proof}

\section{Proofs of Theorems \ref{2xp2}, \ref{1xp2}, \ref{central}, \ref{p/3}, \ref{3xp2}  and \ref{legendre}}\label{SpecialTypeiii}

\begin{proof}[Proof of Theorem \ref{2xp2}]  Notice that
$$  f(x)=2x^{(p+1)/2} \text{ mod } p = 2\left(\frac{x}{p}\right)x \text{ mod } p. $$
Since $p=1$ mod $4$ the quadratic residue property gives us
\be \label{residuecondition} \left(\frac{m}{p}\right)=\begin{cases}  +1, & \text{ if $m=1$ mod $4$,} \\  -1, & \text{ if $m=3$ mod $4$,} \\      \end{cases}\ee
for any integer  $m$ with $1\leq m\leq 4t-1$.

Suppose first that $i=(2p-(4t-1)n)/4$.
Since $n\equiv 2$ mod $4$ we are guaranteed that $i=(2p-(4t-1)n)/4$ is an integer, with $i>0$ from the upper bound in \eqref{nrange}, and $i<n$ for $n>2p/(4t+3)$ which certainly follows from the lower bound. Plainly
\be \label{congi} 4i\equiv 2p \text{ mod }n. \ee

The lower bound in \eqref{nrange} is to ensure that  $i+2tn \geq p$ so that the elements $x$ of $I_i$ can be written
$ x=i+\ell n$ with $0\leq \ell \leq 2t-1$, where $n<2p/(4t-3)$ from the upper bound ensures that $i+(2t-1)n<p$.  Writing $x=i+\ell n$ we have
$$ \left(\frac{x}{p}\right) = \left(\frac{4x}{p} \right) = \left(\frac{-(4t-1)n+4\ell n}{p}\right)=\left(\frac{n}{p}\right) \left(\frac{4\ell -(4t-1)}{p}\right). $$
Since $\left(\frac{-1}{p}\right)=1$ we get from \eqref{residuecondition} that
$$\left(\frac{x}{p}\right) =\left(\frac{n}{p}\right) \left(\frac{4(t-\ell)-1}{p}\right)=-\left(\frac{n}{p}\right), \;\; \ell=0,\ldots ,t-1, $$
and
$$\left(\frac{x}{p}\right) =\left(\frac{n}{p}\right) \left(\frac{4(\ell-t)+1}{p}\right)=\left(\frac{n}{p}\right), \;\; \ell =t,\ldots ,2t-1. $$
Notice that $2x<p$ iff $4\ell<4t-1$, that is $0<2x<p$ for $0\leq \ell \leq t-1$ and $p<2x<2p$ for $\ell=t,\ldots ,2t-1.$
Hence
$$ \left(\frac{n}{p}\right)=-1 \;\; \Rightarrow \;\;  f(x) = \begin{cases} 2x, & \text{ if $\ell=0,\ldots ,t-1$,}  \\ 2p-2x, & \text{ if $\ell=t,\ldots 2t-1$,}\end{cases}$$
while
$$\left(\frac{n}{p}\right)=1 \;\; \Rightarrow \;\;
 f(x) = \begin{cases} p-2x, & \text{ if $\ell=0,\ldots ,t-1$,}  \\ 2x-p, & \text{ if $\ell=t,\ldots ,2t-1$.}\end{cases}$$
Condition \eqref{congi} ensures that $f(x)$  takes the value $j$ mod $n$ in \eqref{defi} for all  $x$ in $I_i$.

Similarly for $i=(2p-(4t-3)n)/4$ we have $0< i<n$ and $I_j=\{x=i+\ell n\; :\; 0\leq \ell \leq 2t-1\}$ for $2p/(4t+1)<n<2p/(4t-3)$
with $2x<p$ iff $\ell \leq t-1$. This time
$$ \left(\frac{x}{p} \right)= \left(\frac{n}{p}\right) \left(\frac{4t-3-4\ell}{p}\right) =\left(\frac{n}{p}\right), \;\; \ell=0,\ldots ,t-1,$$
and
$$ \left(\frac{x}{p}\right) = \left(\frac{n}{p}\right) \left(\frac{4\ell -4t+3}{p}\right) =-\left(\frac{n}{p}\right), \;\; \ell=t,\ldots ,2t-1,$$
giving the same forms for $f(x)$, but with the role of $\left(\frac{n}{p}\right)=1$ or $-1$ reversed.
\end{proof}

\begin{proof}[Proof of Theorem \ref{1xp2}]  Let $f(x)= \left(\frac np\right) x^{(p+1)/2}$.
Since $p\equiv 1$ mod $8$ the quadratic residue condition says that $\left(\frac{\ell}{p}\right)=1$ for $1\leq \ell\leq t$.
We have $I_0=\{\ell n\; : \; 1\leq \ell \leq \lfloor (p-1)/n\rfloor\}$, where  $\lfloor (p-1)/n\rfloor\leq t$. Hence for the $x=n\ell$ in $I_0$ we have
 $ f(x)\equiv \left(\frac{n}{p}\right) \left(\frac{n\ell}{p}\right) x  \text{ mod } p =x. $
\end{proof}

\begin{proof}[Proof of Theorem \ref{central}]  Observe that $\left( \frac{ \frac{1}{2}(p-1)-i}{p}\right)=\left(\frac{2}{p}\right) \left(\frac{2i+1}{p}\right)$, reducing  the consecutive residues or nonresidues about $p/2$  condition  to $\left(\frac{2i-1}{p}\right)=1$ all $1\leq i \leq T$.

For $i=an-(n/2-1+\delta/2)p=p(1-\delta/2)-n(t-1)/2$, where $\delta =0$ for $n$ even and $1$ for $n$ odd, we have $i>0$ for $n<(2-\delta)p/(t-1)$, and $i<n$ for $n>(2-\delta)p/(t+1)$. We also have $i+nt=p(1-\delta/2)+n(t+1)/2>p$ (immediately for $n$ even and from $n>p/(t+1)$ for $n$ odd).

Hence $x$ in $I_i$ can be written $i+\ell n$ with $0\leq \ell <t$ and 
\be  \left(\frac{x}{p}\right)= \left(\frac{an+\ell n}{p}\right) =\left( \frac{n}{p}\right) \left(\frac{a+\ell}{p}\right)=\left( \frac{n}{p}\right) \left( \frac{2}{p}\right).   \qedhere \ee

\end{proof}

\begin{proof}[Proof of Theorem \ref{p/3}] Since $p\equiv 1$ mod $4$ we have $\left(\frac{ \frac{1}{3}(p-\delta) \pm i}{p}\right)= \left(\frac{3}{p}\right)\left(\frac{3i\mp \delta}{p}\right)$ and the equivalent form \eqref{equiv} is plain.
Writing $\ve=1,2$ or $3$ as $n\equiv \ve$ mod $3$, suppose that  $i=na_1-(n-\ve)p/3=\ve p/3 -n(3T_1-3+\delta)/3$
and $i>0$ for $n<\ve p /(3T_1-3+\delta)$ and $i<n$ for $n> \ve p/(3T_1+\delta)$.
We also have  $i+(T_1+T_2)n=\ve p/3 +(T_2+1-\delta/3)n\geq p$ automatically for $\ve =3$, and for $\ve =1$ or $2$ if  $n\geq (3-\ve)p/(3T_2+3-\delta)$ which follows from $n>\ve p/(3T_1+\delta)$ for $\ve T_2 +\ve \geq \delta +(3-\ve)T_1$.
Hence $x$ in $I_i$ can be written $x=i+\ell n$ with $0\leq \ell < T_1+T_2$ and $\left(\frac{x}{p}\right)=\left(\frac{na_1+\ell n}{p}\right)=\left(\frac{n}{p}\right)\left(\frac{a_1+\ell}{p}\right)=\left(\frac{3n}{p}\right)$.
The remaining cases are similar with $i=\ve' p/3 -n(T_2-\delta/3)$ where $\ve'=3,2,1$ as $n\equiv 0,1$ or $2$ mod $3$.
\end{proof}

\begin{proof}[Proof of Theorem \ref{3xp2}]  For $i=na-(s-u)p$ the upper bound $n\leq (s-u)p/(a-1)$  in \eqref{nrange}  ensures that $i\leq n$, and the lower bound $n>(s-u)p/a$ that  $i>0$.
From $n\geq (s-u+1)p/(a+t)$ we also have  $i+tn \geq p,$ so that the elements of $I_i$
can be written $x=i+\ell n$ with $0\leq \ell <t$ and
$$ \left(\frac{x}{p}\right)=\left(\frac{i+\ell n}{p}\right)= \left(\frac{na+\ell n}{p}\right)=\left(\frac{n}{p}\right) \left(\frac{a+\ell}{p}\right)=\left(\frac{na}{p}\right) . $$

Since the gap between our upper and lower bounds in \eqref{nrange}  is
$$\min\left\{ \frac{(s+1-r-u(t+1))}{(a-1)(a+t)},\; \frac{(s-u)}{a(a-1)}\right\}p> \frac{p}{(a+t)^2},$$
we are guaranteed an $n$ if $(a+t)<\sqrt{p}$.

The proof for $i=(s-u+1)p-n(a+t-1)$ is similar.
\end{proof}

\comments{Similarly, if $ i=(s-u+1)p-n(a+t-1)$ and 
\be \label{nrange2}\frac{(s+1-u)p}{a+t} \leq n \leq  \min\left\{ \frac{s-u}{a-1}, \; \frac{s+1-u}{a+t-1}\right\} p, \ee
for some integer $s+1\geq u > s+1-\frac{1}{2}(a+t)$, then \eqref{deff} is the identity on $I_i$.

For $i=(s-u+1)p-n(a+t-1)$ the lower bound in \eqref{nrange} ensures that $i\leq n$, the upper bound  $n\leq (s-u)p/(a-1)$ that $i+tn\geq p$,
and $n\leq (s-u+1)p/(a+t-1)$ that $i> 0$. So
again $x=i+\ell n$ with $0\leq \ell <t$ and
$$\left(\frac{x}{p}\right)=\left(\frac{i+\ell n}{p}\right)= \left(\frac{n(a+t-1)-\ell n}{p}\right)=\left(\frac{n}{p}\right)\left(\frac{a+(t-1-\ell)}{p}\right)=\left(\frac{na}{p}\right) . $$
Writing $f(x)=\left(\frac{an}{p}\right)\left(\frac{x}{p}\right) x$ mod $p$ we have $f(x)=x$ on the two specified $I_i$.

Since the gap between our upper and lower bounds in \eqref{nrange}  is
$$\min\left\{ \frac{(s+1-r-u(t+1))}{(a-1)(a+t)},\frac{(s+1-u)}{(a+t-1)(a+t)}\right\}p> \frac{p}{(a+t)^2},$$
we are guaranteed an $n$ if $(a+t)<\sqrt{p}$.}

\begin{proof}[Proof of  Theorem  \ref{thirds}]  As in the proof of Theorem \ref{central} we can write any $x$ in $I_i$ in the form $x=i+\ell n\equiv an + \ell n $ mod $p$ for $0\leq \ell <t$, and so by \eqref{cubicruns},
$$ x^{(p-1)/3} \equiv (an+\ell n)^{(p-1)/3} =n^{(p-1)/3}(a+\ell)^{(p-1)/3} \equiv (an)^{(p-1)/3} \text{ mod } p. $$
So for \eqref{f/3} we have $f(x)\equiv (an)^{3(k-1)} x\equiv x$ mod $p$ on $I_i$.
Similarly, for $x$ in $I_i$,  
$$ x^{(p-1)/6} \equiv n^{(p-1)/6}(a+\ell)^{(p-1)/6} \equiv \pm  (an)^{(p-1)/6} \text{ mod } p, $$
and for \eqref{f/6} we have $f(x)\equiv \pm (an)^{3(k-1)}x \equiv \pm x$ mod $p$ on $I_i$, where $p-i\equiv i$ mod $n$ for $n$ odd.
\end{proof}

\begin{proof}[Proof of  Theorem  \ref{legendre}] Suppose that $p \equiv 1$ mod 4, $p=n+w$ with $2 \le w<n$. The residue classes with two elements  consist of the pairs $I_y=\{y, -(w-y)\}, I_{w-y}=\{w-y, -y\}$, $1\leq y\leq w/2$, with the remaining classes containing one element.
If $\left(\frac{y}{p}\right),\left(\frac{w-y}{p} \right)$ both equal $1$ then $f(x)=x^{(p+1)/2}$ fixes these sets;  if both equal $-1$ it
switches the pair.
\end{proof}

\end{document}